\documentclass[11pt,a4paper]{amsart}
\usepackage{amsmath, amsthm, amsfonts, amssymb}
\usepackage[all,cmtip]{xy}
\usepackage[bookmarks=false]{hyperref}

\usepackage{color}

\newtheorem{letterthm}{Theorem}

\newtheorem{lettercor}[letterthm]{Corollary}

\newtheorem{thm}{Theorem}[section]
\newtheorem{lem}[thm]{Lemma}

\newtheorem{prop}[thm]{Proposition}

\theoremstyle{definition}

\newtheorem{df}[thm]{Definition}

\newtheorem*{claim}{Claim}

\newcommand{\R}{\mathbf{R}}
\newcommand{\C}{\mathbf{C}}
\newcommand{\Z}{\mathbf{Z}}
\newcommand{\F}{\mathbf{F}}
\newcommand{\Q}{\mathbf{Q}}
\newcommand{\N}{\mathbf{N}}

\newcommand{\B}{\mathbf{B}}
\newcommand{\K}{\mathbf{K}}

\newcommand{\Ad}{\mathord{\text{\rm Ad}}}
\newcommand{\id}{\text{\rm id}}

\newcommand{\rL}{\mathord{\text{\rm L}}}
\newcommand{\rC}{\mathord{\text{\rm C}}}

\newcommand{\rE}{\mathord{\text{\rm E}}}

\newcommand{\core}{\mathord{\text{\rm c}}}

\newcommand{\alg}{\mathord{\text{\rm alg}}}

\newcommand{\Prob}{\mathord{\text{\rm Prob}}}
\newcommand{\SL}{\mathord{\text{\rm SL}}}
\newcommand{\Leb}{\mathord{\text{\rm Leb}}}

\newcommand{\gr}{\mathord{\text{\rm gr}}}

\newcommand{\spn}{\mathord{\text{\rm span}}}
\newcommand{\ovt}{\mathbin{\overline{\otimes}}}

\newcommand{\otm}{\otimes_{\text{\rm min}}}

\newcommand{\ota}{\otimes_{\text{\rm alg}}}
\newcommand{\red}{\rtimes_{\text{\rm red}}}

\begin{document}

\title[Bi-exact groups, strongly ergodic actions and type III factors]{Bi-exact groups, strongly ergodic actions and group measure space type III factors with no central sequence}

\begin{abstract}
We investigate the asymptotic structure of (possibly type ${\rm III}$) crossed product von Neumann algebras $M = B \rtimes \Gamma$ arising from arbitrary actions $\Gamma \curvearrowright B$ of bi-exact discrete groups (e.g.\ free groups) on amenable von Neumann algebras. We prove a spectral gap rigidity result for the central sequence algebra $N' \cap M^\omega$ of any nonamenable von Neumann subalgebra with normal expectation $N \subset M$. We use this result to show that for any strongly ergodic essentially free nonsingular action $\Gamma \curvearrowright (X, \mu)$ of any bi-exact countable discrete group on a standard probability space, the corresponding group measure space factor $\rL^\infty(X) \rtimes \Gamma$ has no nontrivial central sequence. Using recent results of Boutonnet--Ioana--Salehi Golsefidy \cite{BISG15}, we construct, for every $0 < \lambda \leq 1$, a type ${\rm III_\lambda}$ strongly ergodic essentially free nonsingular action $\F_\infty \curvearrowright (X_\lambda, \mu_\lambda)$ of the free group $\mathbf F_\infty$ on a standard probability space so that the corresponding group measure space type ${\rm III_\lambda}$ factor $\rL^\infty(X_\lambda, \mu_\lambda) \rtimes \F_\infty$ has no nontrivial central sequence by our main result. In particular, we obtain the first examples of group measure space type ${\rm III}$ factors with no nontrivial central sequence.
\end{abstract}

\author{Cyril Houdayer}
\address{Laboratoire de Math\'ematiques d'Orsay, Universit\'e Paris-Sud, CNRS, Universit\'e Paris-Saclay, 91405 Orsay, France}
\email{cyril.houdayer@math.u-psud.fr}
\thanks{CH is supported by ERC Starting Grant GAN 637601}

\author{Yusuke Isono}
\address{RIMS, Kyoto University, 606-8502 Kyoto, Japan}
\email{isono@kurims.kyoto-u.ac.jp}
\thanks{YI is supported by JSPS Research Fellowship}


\maketitle

\section{Introduction and statement of the main results}

The {\em group measure space construction} of Murray and von Neumann \cite{MvN43} associates to any ergodic (essentially) free nonsingular action $\Gamma \curvearrowright (X, \mu)$ of a countable discrete group on a standard probability space a factor denoted by $\rL^\infty(X) \rtimes \Gamma$. A fundamental question in operator algebras is how much information does the group measure space factor  $\rL^\infty(X) \rtimes \Gamma$ retain from the group action $\Gamma \curvearrowright (X, \mu)$? This question has attracted a lot of attention during the last 15 years and several important developments regarding the structure and the rigidity of group measure space factors have been made possible thanks to Popa's {\em deformation/rigidity} theory \cite{Po06a}. We refer the reader to \cite{Ga10, Va10, Io12b} for recent surveys on this topic.

One of the questions we address in this paper is the following general problem: Under which assumptions on the countable discrete group $\Gamma$ and the ergodic free nonsingular action $\Gamma \curvearrowright (X, \mu)$, the group measure space factor $\rL^\infty(X) \rtimes \Gamma$ is full? Recall from \cite{Co74} that a factor $M$ with separable predual is {\em full} if its asymptotic centralizer $M_\omega$ is trivial for some (or any) nonprincipal ultrafilter $\omega \in \beta(\N) \setminus \N$. By \cite[Theorem 5.2]{AH12}, a factor $M$ with separable predual is full if and only if its central sequence algebra $M' \cap M^\omega$ is trivial for some (or any) nonprincipal ultrafilter $\omega \in \beta(\N) \setminus \N$ (see Section \ref{preliminaries} for further details). If the group measure space factor $\rL^\infty(X) \rtimes \Gamma$ is full then the free nonsingular action $\Gamma \curvearrowright (X, \mu)$ is necessarily {\em strongly ergodic}, that is, any $\Gamma$-asymptotically invariant sequence of measurable subsets of $X$ is trivial. The converse is not true in general as demonstrated in the celebrated example by Connes and Jones \cite{CJ81}. Indeed, they exhibited an example of a strongly ergodic free probability measure preserving (pmp) action such the associated group measure space ${\rm II_1}$ factor is {\em McDuff}, that is, tensorially absorbs the hyperfinite ${\rm II_1}$ factor of Murray and von Neumann.

The general problem mentioned above has nevertheless a satisfactory answer in the case when the action $\Gamma \curvearrowright (X, \mu)$ is {\em pmp}. Indeed, it was shown by Choda in \cite{Ch81} that when the countable discrete group $\Gamma$ is not {\em inner amenable} and the  free pmp action $\Gamma \curvearrowright (X, \mu)$ is strongly ergodic, then the group measure space ${\rm II_1}$ factor $\rL^\infty(X) \rtimes \Gamma$ is full. The facts that the group $\Gamma$ is not inner amenable and the action $\Gamma \curvearrowright (X, \mu)$ is pmp imply that all the central sequences in $\rL^\infty(X) \rtimes \Gamma$ must asymptotically lie in $\rL^\infty(X)$. It follows immediately that $\rL^\infty(X) \rtimes \Gamma$ is full if the action is strongly ergodic. In the above reasoning, the assumption that the action $\Gamma \curvearrowright (X, \mu)$ is pmp is crucial since nonamenable (and in particular non-inner amenable) groups always admit an amenable (in the sense of Zimmer \cite[Definition 4.3.1]{Zi84}) type ${\rm III}$ ergodic nonsingular action, namely the Poisson boundary action. Very little is known about the general problem mentioned above when the action $\Gamma \curvearrowright (X, \mu)$ is no longer pmp and is more generally nonsingular (possibly of type ${\rm III}$).

In this paper, we investigate the asymptotic structure of (possibly type ${\rm III}$) group measure space factors $\rL^\infty(X) \rtimes \Gamma$ and more generally of (possibly type ${\rm III}$) crossed product von Neumann algebras $B \rtimes \Gamma$ arising from arbitrary actions $\Gamma \curvearrowright B$ of bi-exact discrete groups on amenable von Neumann algebras. The class of {\em bi-exact} discrete groups was introduced by Ozawa in \cite{Oz04} (see also \cite[Chapter 15]{BO08}) and includes amenable groups, free groups, Gromov word-hyperbolic groups and discrete subgroups of connected simple Lie groups of real rank one. We refer the reader to Section \ref{preliminaries} for a precise definition. Any bi-exact discrete group is either amenable or non-inner amenable \cite{Oz04}. Ozawa's celebrated result \cite{Oz03} asserts that bi-exact discrete groups $\Gamma$ give rise to {\em solid} group von Neumann algebras $\rL(\Gamma)$, that is, for any diffuse von Neumann algebra $A \subset \rL(\Gamma)$, the relative commutant $A' \cap \rL(\Gamma)$ is amenable. Moreover, any solid ${\rm II_1}$ factor is either amenable or full \cite[Proposition 7]{Oz03}. Recall that an inclusion of von Neumann algebras $N \subset M$ is {\em with expectation} if there exists a faithful normal conditional expectation $\rE_N : M \to N$.

Our first main result is a {\em spectral gap rigidity} result inside crossed product von Neumann algebras $M = B \rtimes \Gamma$ arising from arbitrary actions $\Gamma \curvearrowright B$ of bi-exact discrete groups on amenable $\sigma$-finite von Neumann algebras. More precisely, we prove that for any von Neumann subalgebra with expectation $N \subset M$, either $N$ has a nonzero amenable direct summand or the central sequence algebra $N' \cap M^\omega$ lies in the smaller algebra $B^\omega \rtimes \Gamma$. Our Theorem \ref{thmA} can be regarded as an analogue of the spectral gap rigidity results discovered by Peterson in \cite[Theorem 4.3]{Pe06} and Popa in \cite[Theorem 1.5]{Po06b} and \cite[Lemma 2.2]{Po06c}. 

\begin{letterthm}\label{thmA}
	Let $\Gamma$ be any bi-exact discrete group, $B$ any amenable $\sigma$-finite von Neumann algebra and $\Gamma \curvearrowright B$ any action. Denote by $M:=B\rtimes \Gamma$ the corresponding crossed product von Neumann algebra. Let $p\in M$ be any nonzero projection and $N\subset pMp$ any von Neumann subalgebra with expectation. Let $\omega \in \beta(\N) \setminus \N$ be any nonprincipal ultrafilter.
	
Then at least one of the following conditions holds true:
	\begin{itemize}
		\item The von Neumann algebra $N$ has a nonzero amenable direct summand.
		\item We have $N'\cap pM^\omega p \subset p(B^\omega \rtimes \Gamma) p$. In this case, we further obtain $A \preceq_{B^\omega \rtimes \Gamma} B^\omega$ for any finite von  Neumann subalgebra with expectation $A\subset N'\cap pM^\omega p$.
	\end{itemize}
\end{letterthm}

We refer the reader to Section \ref{preliminaries} for ultraproduct von Neumann algebras and Popa's intertwining techniques inside arbitrary von Neumann algebras. The proof of Theorem \ref{thmA} given in Section \ref{section-thmA} (see Theorems \ref{theorem for thmA 1} and \ref{theorem for thmA 2}) uses a combination of Ozawa's $\rC^*$-algebraic techniques \cite{Oz03, Oz04, Is12}, ultraproduct von Neumann algebraic techniques \cite{Oc85, AH12} and the recent generalization of Popa's intertwining-by-bimodules to the framework of type ${\rm III}$ von Neumann algebras developed by the authors in \cite{HI15}. The interesting feature of the proof of Theorem \ref{thmA} is that it does {\em not} rely on Connes--Tomita--Takesaki modular theory. Indeed, unlike other instances of Popa's spectral gap rigidity results in the literature which typically rely on using amenable traces and hence require the ambient von Neumann algebra to be (semi)finite, we use instead unital completely positive (ucp) maps and exploit Ozawa's $\rC^*$-algebraic techniques \cite{Oz03, Oz04} to prove the existence of norm one projections. The main advantage of this approach is that it allows us to work directly inside the (possibly type ${\rm III}$) crossed product von Neumann algebra $M = B \rtimes \Gamma$ without appealing to the continuous core decomposition. In this respect, our approach is similar to the one we developed in our previous paper \cite{HI15}. We refer the reader to \cite{HR14, HU15, HV12, Is12, Is13} for other structural/rigidity results for type ${\rm III}$ factors involving the continuous core decomposition.

Following \cite{HR14, Oz04}, we say that a von Neumann algebra $M$ is $\omega$-{\em semisolid} if for any von Neumann subalgebra $N\subset M$ with expectation such that the relative commutant $N'\cap M^\omega$ has no type $\rm I$ direct summand, we have that $N$ is amenable. The next corollary strengthens the indecomposability properties of crossed product von Neumann algebras $B \rtimes \Gamma$ arising from arbitrary actions $\Gamma \curvearrowright B$ of bi-exact discrete groups on {\em abelian} von Neumann algebras (see \cite{Oz04, HV12, Is12} for previous results).

\begin{lettercor}\label{corB}
	Let $\Gamma$ be any bi-exact discrete group, $B$ any abelian $\sigma$-finite von Neumann algebra and $\Gamma \curvearrowright B$ any action. Let $\omega \in \beta(\N) \setminus \N$ be any nonprincipal ultrafilter. Then the crossed product von Neumann algebra $B\rtimes \Gamma$ is $\omega$-semisolid. 
	
In particular, if $B\rtimes \Gamma$ is a nonamenable factor, then $B\rtimes \Gamma$ is {\em prime}, that is, $B\rtimes \Gamma$ cannot be written as a tensor product $Q_1 \ovt Q_2$ of diffuse factors.
\end{lettercor}

Our second main result, Theorem \ref{thmC} below, is an answer to the general problem mentioned earlier in the case when the acting group is bi-exact. Indeed, using Theorem \ref{thmA} in the case when the action $\Gamma \curvearrowright B$ arises from a strongly ergodic free nonsingular action $\Gamma \curvearrowright (X, \mu)$ of a bi-exact countable discrete group on a standard probability space, we show that the group measure space factor $\rL^\infty(X) \rtimes \Gamma$ is full.

\begin{letterthm}\label{thmC}
Let $\Gamma$ be any bi-exact countable discrete group and $\Gamma \curvearrowright (X,\mu)$ any strongly ergodic free nonsingular action on a standard probability space. Then the group measure space factor $\rL^\infty(X)\rtimes \Gamma$ is full.
\end{letterthm}

The proof of Theorem \ref{thmC} uses a combination of Theorem \ref{thmA} and the useful Lemma \ref{lem-strong-ergodicity} below which proves the existence of a nontrivial centralizing sequence $(u_n)_n$ in every nonfull factor $M = \rL(\mathcal R)$ arising from a strongly ergodic nonsingular equivalence relation $\mathcal R$ defined on a standard probability space such that $(u_n)_n$ ``does not embed" into the Cartan subalgebra $\rL^\infty(X)$. Our Lemma 5.1 is a nonsingular generalization of a recent result of Hoff (see the first part of the proof of \cite[Proposition C]{Ho15}). In view of Choda's result \cite{Ch81}, we do not know whether Theorem \ref{thmC} holds true more generally for (arbitrary strongly ergodic free nonsingular actions of) arbitrary non-inner amenable groups instead of bi-exact groups. We point out that Ozawa recently showed in \cite{Oz16} that Theorem \ref{thmC} holds true for arbitrary strongly ergodic free nonsingular actions of $\SL_3(\Z)$, which is not bi-exact by \cite{Sa09}.

We finally exploit recent results of Boutonnet--Ioana--Salehi Golsefidy \cite{BISG15} to construct, for every $0 < \lambda \leq 1$, examples of type ${\rm III_\lambda}$ strongly ergodic free nonsingular actions of a free group on a standard probability space. It is shown in \cite[Theorem A]{BISG15} that for any (not necessarily compact) connected simple Lie group and any countable dense subgroup $\Lambda < G$ with ``algebraic entries" (e.g.\ $(\Lambda < G) = (\SL_n(\Q) < \SL_n(\R))$ for $n \geq 2$), the left translation action $\Lambda \curvearrowright G$ is strongly ergodic. By taking a suitable non-unimodular closed subgroup $P < G$, the quotient action $\Lambda \curvearrowright G/P$ is still strongly ergodic and of type ${\rm III}$. The quotient action $\Lambda \curvearrowright G/P$ need not be essentially free in general. However, using a ``direct product" construction similar to the one used in \cite[Corollary B]{HV12}, we can then construct strongly ergodic essentially free nonsingular actions of free groups and we obtain the following corollary.

\begin{lettercor}\label{corD}
For every $0 < \lambda \leq 1$, there exists a strongly ergodic free nonsingular action $\F_\infty \curvearrowright (X_\lambda, \mu_\lambda)$ of type ${\rm III_\lambda}$ so that the
 group measure space factor $\rL^\infty(X_\lambda, \mu_\lambda) \rtimes \F_\infty$ is of type ${\rm III_\lambda}$ and is full.
 
Moreover, there exists a strongly ergodic free nonsingular action $\F_\infty \curvearrowright (X_\infty, \mu_\infty)$ of type ${\rm II_\infty}$ so that the
 group measure space factor $\rL^\infty(X_\infty, \mu_\infty) \rtimes \F_\infty$ is of type ${\rm II_\infty}$ and is full.
\end{lettercor}

The first examples of full factors of type ${\rm III}$ were discovered by Connes in \cite{Co74}. He showed that the factors $$M_{n,k, \varphi} = \left( \overline{\bigotimes}_{\F_n} (\mathbf M_k(\C), \varphi) \right) \rtimes \F_n$$ 
arising from Connes--St\o rmer Bernoulli shifts of free groups $\F_n \curvearrowright \overline{\bigotimes}_{\F_n} (\mathbf M_k(\C), \varphi)$ are full if $n, k \geq 2$ and of type ${\rm III}$ if $\varphi$ is not tracial. Observe that the factors $M_{n,k, \varphi}$ possess $\overline{\bigotimes}_{\F_n} \C^k$ as a Cartan subalgebra. This implies that the underlying ergodic nonsingular equivalence relation is strongly ergodic \cite{FM75}. However, Connes--St\o rmer Bernoulli crossed products need not be $\ast$-isomorphic to group measure space factors. In this respect, Corollary \ref{corD} above provides the first class of group measure space type ${\rm III}$ factors with no nontrivial central sequence.

We finally point out that the group measure space type ${\rm III}$ factors in Corollary \ref{corD} possess a unique Cartan subalgebra, up to unitary conjugacy, by \cite[Theorem A]{HV12} (see \cite{PV11} for the trace preserving case). Moreover, Corollary \ref{corD} provides new examples of group measure type ${\rm III}$ factors with a unique Cartan subalgebra, up to unitary conjugacy. Indeed, the examples of ergodic free nonsingular actions considered in \cite[Corollary B]{HV12} are not strongly ergodic since they have an amenable action as a quotient. In particular, the group measure space type ${\rm III}$ factors in \cite[Corollary B]{HV12} are not full while the ones in Corollary \ref{corD} are full.

\subsection*{Acknowledgments} It is our pleasure to thank Adrian Ioana, Dimitri Shlyakhtenko, Yoshimichi Ueda and Stefaan Vaes for their valuable comments.

\tableofcontents

\section{Preliminaries}\label{preliminaries}

For any von Neumann algebra $M$, we will denote by $\mathcal Z(M)$ the centre of $M$, by $\mathcal U(M)$ the group of unitaries in $M$ and by $(M, \rL^2(M), J^M, \mathfrak P^M)$ a standard form for $M$. We will say that an inclusion of von Neumann algebras $P \subset 1_P M 1_P$ is {\em with expectation} if there exists a faithful normal conditional expectation $\rE_P : 1_P M 1_P \to P$. 

\subsection*{Crossed product von Neumann algebras}

We will use the following terminology and notation regarding crossed product von Neumann algebras. 
Let $\Gamma$ be any discrete group, $B$ any $\sigma$-finite von Neumann algebra and $\Gamma \curvearrowright B$ any action. Denote by $M:= B\rtimes \Gamma$ the corresponding {\em crossed product} von Neumann algebra and by $\rE_{B} : M \to B$ the canonical faithful normal conditional expectation given by $\rE_B(b \lambda_g) = \delta_{g, e} b$ for all $g \in \Gamma$ and all $b \in B$. Fix a standard form $(B, \rL^2(B), J^B, \mathfrak P^B)$ for $B$. Denote by $u : \Gamma \to \mathcal U(\rL^2(B))$ the canonical unitary representation implementing the action $\Gamma \curvearrowright B$. A standard form $(M, \rL^2(M), J^M, \mathfrak P^M)$ for $M$ is given by $\rL^2(M) = \rL^2(B) \otimes \ell^2(\Gamma)$ and 
$$J^M (\xi \otimes \delta_g) = u_g^* J^B \xi \otimes \delta_{g^{-1}} \quad \text{for all } \xi  \in \rL^2(B) \text{ and all } g \in \Gamma.$$
The Jones projection $e_B : \rL^2(M) \to \rL^2(B)$ is then simply given by $e_B = 1 \otimes P_{\C \delta_e}$ where $P_{\C \delta_e} : \ell^2(\Gamma) \to \C \delta_e$ is the orthogonal projection onto $\C \delta_e$. For crossed product von Neumann algebras $M = B \rtimes \Gamma$, we will always use such a standard form $(M, \rL^2(M), J^M, \mathfrak P^M)$ as defined above.

\subsection*{Ultraproduct von Neumann algebras}

Let $M$ be any $\sigma$-finite von Neumann algebra and $\omega \in \beta(\N) \setminus \N$ any nonprincipal ultrafilter. Define
\begin{align*}
\mathcal I_\omega(M) &= \left\{ (x_n)_n \in \ell^\infty(M) \mid x_n \to 0\ \ast\text{-strongly as } n \to \omega \right\} \\
\mathcal M^\omega(M) &= \left \{ (x_n)_n \in \ell^\infty(M) \mid  (x_n)_n \, \mathcal I_\omega(M) \subset \mathcal I_\omega(M) \text{ and } \mathcal I_\omega(M) \, (x_n)_n \subset \mathcal I_\omega(M)\right\}.
\end{align*}
The {\em multiplier algebra} $\mathcal M^\omega(M)$ is a C$^*$-algebra and $\mathcal I_\omega(M) \subset \mathcal M^\omega(M)$ is a norm closed two-sided ideal. Following \cite[\S 5.1]{Oc85}, we define the {\em ultraproduct von Neumann algebra} $M^\omega$ by $M^\omega := \mathcal M^\omega(M) / \mathcal I_\omega(M)$, which is indeed known to be a von Neumann algebra. We denote the image of $(x_n)_n \in \mathcal M^\omega(M)$ by $(x_n)^\omega \in M^\omega$. 

For every $x \in M$, the constant sequence $(x)_n$ lies in the multiplier algebra $\mathcal M^\omega(M)$. We will then identify $M$ with $(M + \mathcal I_\omega(M))/ \mathcal I_\omega(M)$ and regard $M \subset M^\omega$ as a von Neumann subalgebra. The map $\rE_\omega : M^\omega \to M : (x_n)^\omega \mapsto \sigma \text{-weak} \lim_{n \to \omega} x_n$ is a faithful normal conditional expectation. For every faithful state $\varphi \in M_\ast$, the formula $\varphi^\omega := \varphi \circ \rE_\omega$ defines a faithful normal state on $M^\omega$. Observe that $\varphi^\omega((x_n)^\omega) = \lim_{n \to \omega} \varphi(x_n)$ for all $(x_n)^\omega \in M^\omega$. 

Following \cite[\S2]{Co74}, we define
$$\mathcal M_\omega(M) := \left\{ (x_n)_n \in \ell^\infty(M) \mid \lim_{n \to \omega} \|x_n \varphi - \varphi x_n\| = 0, \forall \varphi \in M_\ast \right\}.$$
We have $\mathcal I_\omega (M) \subset \mathcal M_\omega(M) \subset \mathcal M^\omega(M)$. The {\em asymptotic centralizer} is defined by $M_\omega := \mathcal M_\omega(M)/\mathcal I_\omega(M)$. We have $M_\omega \subset M^\omega$. Moreover, by \cite[Proposition 2.8]{Co74} (see also \cite[Proposition 4.35]{AH12}), we have $M_\omega = M' \cap (M^\omega)^{\varphi^\omega}$ for every faithful state $\varphi \in M_\ast$.

Let $Q \subset M$ be any von Neumann subalgebra with faithful normal conditional expectation $\rE_Q : M \to Q$. Choose a faithful state $\varphi \in M_\ast$ in such a way that $\varphi = \varphi \circ \rE_Q$. We have $\ell^\infty(Q) \subset \ell^\infty(M)$, $\mathcal I_\omega(Q) \subset \mathcal I_\omega(M)$ and $\mathcal M^\omega(Q) \subset \mathcal M^\omega(M)$. We will then identify $Q^\omega = \mathcal M^\omega(Q) / \mathcal I_\omega(Q)$ with $(\mathcal M^\omega(Q) + \mathcal I_\omega(M)) / \mathcal I_\omega(M)$ and be able to regard $Q^\omega \subset M^\omega$ as a von Neumann subalgebra. Observe that the norm $\|\cdot\|_{(\varphi |_Q)^\omega}$ on $Q^\omega$ is the restriction of the norm $\|\cdot\|_{\varphi^\omega}$ to $Q^\omega$. Observe moreover that $(\rE_Q(x_n))_n \in \mathcal I_\omega(Q)$ for all $(x_n)_n \in \mathcal I_\omega(M)$ and $(\rE_Q(x_n))_n \in \mathcal M^\omega(Q)$ for all $(x_n)_n \in \mathcal M^\omega(M)$. Therefore, the mapping $\rE_{Q^\omega} : M^\omega \to Q^\omega : (x_n)^\omega \mapsto (\rE_Q(x_n))^\omega$ is a well-defined conditional expectation satisfying $\varphi^\omega \circ \rE_{Q^\omega} = \varphi^\omega$. Hence, $\rE_{Q^\omega} : M^\omega \to Q^\omega$ is a faithful normal conditional expectation. For more on ultraproduct von Neumann algebras, we refer the reader to \cite{AH12, Oc85}.

We record the following observation that will be used throughout. Let $\Gamma$ be any discrete group, $B$ any $\sigma$-finite von Neumann algebra and $\Gamma \curvearrowright B$ any action. Put $M := B \rtimes \Gamma$ and denote by $\rE_B : M \to B$ the canonical faithful normal conditional expectation. Choose any faithful state $\varphi \in M_\ast$ such that $\varphi \circ \rE_B = \varphi$. Then the von Neumann subalgebra $B^\omega \vee M \subset M^\omega$ is globally invariant under the modular automorphism group $\sigma^{\varphi^\omega}$ and hence is with expectation. Observe that we have $B^\omega \vee M = B^\omega \rtimes \Gamma$ canonically. Therefore the von Neumann subalgebra $B^\omega \rtimes \Gamma \subset M^\omega$ is with expectation. Denote by $\rE_{B^\omega} : M^\omega \to B^\omega$ and by $\rE_{B^\omega \rtimes \Gamma} : M^\omega \to B^\omega \rtimes \Gamma$ the unique $\varphi^\omega$-preserving conditional expectations. By uniqueness of the $\varphi^\omega$-preserving conditional expectation $\rE_{B^\omega} : M^\omega \to B^\omega$, we have $\rE_{B^\omega} \circ \rE_{B^\omega \rtimes \Gamma} = \rE_{B^\omega}$.

We thank Hiroshi Ando for pointing out to us the following well-known result.

\begin{lem}\label{lem-ultraproduct}
Let $M$ be any $\sigma$-finite von Neumann algebra and $\omega \in \beta(\N) \setminus \N$ any nonprincipal ultrafilter. For any $u \in \mathcal U(M^\omega)$, there exists a sequence $(u_n)_n \in \mathcal M^\omega(M)$ such that $u = (u_n)^\omega$ and $u_n \in \mathcal U(M)$ for every $n \in \N$.
\end{lem}

\begin{proof}
Denote by $f : \mathbf T \to (- \pi, \pi]$ the unique Borel function such that $\exp({\rm i} f(z)) = z$ for all $z \in \mathbf T$. Let $u \in \mathcal U(M^\omega)$ and put $h = f(u) \in M^\omega$. Then $h^* = h$ and $\exp({\rm i} h) = u$. Write $h = (h_n)^\omega$ for some $(h_n)_n \in \mathcal M^\omega(M)$. Since $h^* = h$, up to replacing each $h_n$ by $\frac12(h_n +h_n^*)$, we may assume that $h_n^* = h_n$ for every $n \in \N$. Put $u_n = \exp({\rm i} h_n) \in \mathcal U(M)$ for every $n \in \N$. Since $[-\pi, \pi] \to \mathbf T : t \mapsto \exp({\rm i} t)$ is a continuous function, it is a uniform limit of polynomial functions by Stone--Weierstrass theorem. It follows that $(u_n)_n = (\exp({\rm i} h_n))_n = \exp({\rm i} (h_n)_n) \in \mathcal M^\omega(M)$ and $u = \exp({\rm i} h) = \exp({\rm i} (h_n)^\omega) = (\exp({\rm i} h_n))^\omega = (u_n)^\omega$.
\end{proof}

	We next recall the construction of the Groh--Raynaud ultraproduct. For any Hilbert space $H$, define the {\em ultraproduct Hilbert space} $H_\omega$ as the completion/separation of $\ell^\infty(H)$ with respect to the semi-inner product given by $\langle (\xi_n)_n, (\eta_n)_n \rangle := \lim_{n\to \omega} \langle \xi_n, \eta_n \rangle_H$ for all $(\xi_n)_n, (\eta_n)_n \in \ell^\infty(H)$. We denote the image of $(\xi_n)_n\in \ell^\infty(H)$ by $(\xi_n)_\omega \in H_\omega$. Let $M\subset \B(H)$ be any von Neumann algebra. We define the unital $\ast$-representation $\pi^\omega : \ell^\infty(M) \to \B(H_\omega)$  by 
	$$ \pi^\omega((x_n)_n) (\xi_n)_\omega = ( x_n \xi_n)_\omega \quad \text{for all }  (x_n)_n\in \ell^\infty(\B(H))  \text{ and all }  (\xi_n)_n \in \ell^\infty(H). $$
Let $(M, \rL^2(M), J^M, \mathfrak P^M)$ be a standard form for $M$. The {\em Groh--Raynaud ultraproduct} $N := \prod^\omega M$ is the von Neumann algebra generated by $\pi^\omega(\ell^\infty(M))$. It is known that the inclusion $N \subset \B(\rL^2(M)_\omega)$ is in standard form with  modular conjugation given by $J^N (\xi_n)_\omega := (J^M\xi_n)_\omega$ for all $(\xi_n)_\omega\in \rL^2(M)_\omega$ (see \cite[Corollary 3.9]{Ra99} and \cite[Theorem 3.18]{AH12}). By \cite[Theorem 3.7]{AH12}, the Ocneanu ultraproduct is $\ast$-isomorphic to a corner of the Groh--Raynaud ultraproduct. More precisely, for any faithful state $\varphi \in M_\ast$, denote by $\xi_\varphi \in \mathfrak P^M$ the canonical representing vector. Then the isometry given by
$$w_\varphi : \rL^2(M^\omega) \to \rL^2(M)_\omega : (x_n)^\omega \xi_{\varphi^\omega} \mapsto (x_n\xi_\varphi)_\omega$$
satisfies $w_\varphi^* N w_\varphi = M^\omega$. Define the ultraproduct state $\varphi_\omega = \langle \, \cdot \, (\xi_\varphi)_\omega, (\xi_\varphi)_\omega\rangle \in N_\ast$ and denote by $p \in N$ the support projection of $\varphi_\omega \in N_\ast$. We have $w_\varphi w_\varphi^* = pJ^N p J^N$. Then the condition $w_\varphi^* N w_\varphi = M^\omega$ implies that $pJ^NpJ^N \, N \, p J^NpJ^N\cong pNp\cong M^\omega$ so that the standard representation of $M^\omega$ is given by $\rL^2(M^\omega) = pJ^N pJ^N \, \rL^2(M)_\omega$ with modular conjugation $J^{M^\omega} = p J^N p$.

\begin{lem}\label{lemma for AO in ultraproduct}
	Let $B\subset M$ be any inclusion of $\sigma$-finite von Neumann algebras with faithful normal conditional expectation $\rE_B : M \to B$. Denote by $e_B : \rL^2(M) \to \rL^2(B)$ the corresponding Jones projection. Denote by $N := \prod^\omega M$ the Groh--Raynaud ultraproduct. Let $\varphi \in M_\ast$ be any faithful state such that $\varphi\circ \rE_B=\varphi$ and denote by $p \in N$ the support projection of $\varphi_\omega \in N_\ast$.
	 
Then $\pi^\omega( (e_B)_n )$ commutes with $p$ and $J^N$.
\end{lem}
\begin{proof}
Since $e_B$ commutes with $J^M$, $\pi^\omega( (e_B)_n )$ commutes with $J^N$. Denote by $\xi_\varphi \in \mathfrak P^M$ the canonical vector representing $\varphi \in M_\ast$. Since $e_B M e_B = B e_B$ and $e_B \xi_\varphi = \xi_\varphi$, we have
\begin{eqnarray*}
	\pi^\omega( (e_B)_n )J^N \pi^\omega(\ell^\infty(M))(\xi_\varphi)_\omega 
	&=&J^N \pi^\omega( (e_B)_n ) \pi^\omega(\ell^\infty(M))\pi^\omega( (e_B)_n ) (\xi_\varphi)_\omega \\
	&=& J^N \pi^\omega(\ell^\infty(B))\pi^\omega( (e_B)_n ) (\xi_\varphi)_\omega \\
	&\subset& J^N \pi^\omega(\ell^\infty(M)) (\xi_\varphi)_\omega.
\end{eqnarray*}
Since $p$ is the projection onto the closure of $J^N \pi^\omega(\ell^\infty(M)) (\xi_\varphi)_\omega$, we obtain that $\pi^\omega( (e_B)_n )$ commutes with $p$.
\end{proof}

\subsection*{Equivalence relations and von Neumann algebras}

\begin{df}[\cite{FM75}]
Let $(X, \mu)$ be any standard probability space. A {\em nonsingular equivalence relation} $\mathcal R$ {\em defined on} $(X, \mu)$ is an equivalence relation $\mathcal R \subset X \times X$ which satisfies the following three conditions:
\begin{itemize}
\item[$\rm(i)$]  $\mathcal R \subset X \times X$ is a Borel subset,
\item[$\rm(ii)$]  $\mathcal R$ has countable classes and
\item[$\rm(iii)$]  for every $\varphi \in [\mathcal R]$, we have $[\varphi_\ast \mu] = [\mu]$ where $[\mathcal R]$ denotes the {\em full group} of $\mathcal R$ consisting in all the Borel automorphisms $\varphi : X \to X$ such that $\gr(\varphi) \subset \mathcal R$.
\end{itemize}
\end{df}

Following \cite{FM75}, to any nonsingular equivalence relation $\mathcal R$ defined on a standard probability space $(X, \mu)$, one can associate a von Neumann algebra $M = \rL(\mathcal R)$ which contains $A = \rL^\infty(X)$ as a {\em Cartan subalgebra}, that is, $A \subset M$ is maximal abelian with expectation and the group of normalizing unitaries $\mathcal N_M(A) = \{u \in \mathcal U(M) : u A u^*\}$ generates $M$ as a von Neumann algebra.

When $\Gamma \curvearrowright (X, \mu)$ is a nonsingular Borel action of a countable discrete group on a standard probability space, we will denote by $\mathcal R(\Gamma \curvearrowright X)$ the nonsingular {\em orbit equivalence relation} defined by 
$$\mathcal R(\Gamma \curvearrowright X) = \{(x, \gamma x) \mid \gamma \in \Gamma, x \in X\}.$$
When the action $\Gamma \curvearrowright (X, \mu)$ is moreover {\em essentially free}, there is a canonical isomorphism of pairs of von Neumann algebras
$$\left( \rL^\infty(X) \subset \rL(\mathcal R(\Gamma \curvearrowright X)) \right) \cong \left(\rL^\infty(X) \subset \rL^\infty(X) \rtimes \Gamma \right).$$ 
For more information on nonsingular equivalence relations and their von Neumann algebras, we refer the reader to \cite{FM75}.

\subsection*{Strongly ergodic actions and full factors}

We first recall the concept of {\em strong ergodicity} for group actions and equivalence relations.

\begin{df}
Let $(X, \mu)$ be any standard probability space.
\begin{itemize}
\item[$\rm(i)$]  Let $\Gamma$ be any countable discrete group and $\Gamma \curvearrowright (X, \mu)$ any  ergodic nonsingular action. The action $\Gamma \curvearrowright (X, \mu)$ is said to be {\em strongly ergodic} if for any sequence $(C_n)_n$ of measurable subsets of $X$ such that $\lim_n \mu(\gamma C_n \triangle C_n) = 0$ for all $\gamma \in \Gamma$, we have $\lim_n \mu(C_n)(1 - \mu(C_n)) = 0$.
\item[$\rm(ii)$]  Let $\mathcal R$ be any ergodic nonsingular equivalence relation defined on $(X, \mu)$. The equivalence relation $\mathcal R$ is said to be {\em strongly ergodic} if for any sequence $(C_n)_n$ of measurable subsets of $X$ such that $\lim_n \mu(g C_n \triangle C_n) = 0$ for all $g \in [\mathcal R]$, we have $\lim_n \mu(C_n)(1 - \mu(C_n)) = 0$.
\end{itemize}
\end{df}

Put $A = \rL^\infty(X)$ and fix any nonprincipal ultrafilter $\omega \in \beta(\N) \setminus \N$. Then the ergodic nonsingular action $\Gamma \curvearrowright (X, \mu)$ is strongly ergodic if and only if the ultraproduct action $\Gamma \curvearrowright A^\omega$ defined by $\gamma \cdot (a_n)^\omega = (\gamma \cdot a_n)^\omega$ is ergodic, that is, $(A^\omega)^\Gamma = \C 1$. Likewise, the ergodic nonsingular equivalence relation $\mathcal R$ defined on $(X, \mu)$ is strongly ergodic if and only if $\rL(\mathcal R)' \cap A^\omega = \C 1$. We also have that the nonsingular action $\Gamma \curvearrowright (X, \mu)$ is strongly ergodic if and only if the nonsingular orbit equivalence relation $\mathcal R(\Gamma \curvearrowright X)$ is strongly ergodic.

Following \cite{Co74}, we say that a factor $M$ with separable predual is {\em full} if $M_\omega = \C 1$ for some (or any) nonprincipal ultrafilter $\omega \in \beta(\N) \setminus \N$. By \cite[Theorem 5.2]{AH12}, $M$ is full if and only if $M' \cap M^\omega = \C 1$ for some (or any) nonprincipal ultrafilter $\omega \in \beta(\N) \setminus \N$. Observe that for any ergodic nonsingular equivalence relation $\mathcal R$ defined on a standard probability space $(X, \mu)$, if $\rL(\mathcal R)$ is full then $\mathcal R$ is strongly ergodic.

Connes proved in \cite[Theorem 2.12]{Co74} that factors of type ${\rm III_0}$ are never full. Ueda showed in \cite[Corollary 11]{Ue00} that ergodic nonsingular equivalence relations of type ${\rm III_0}$ are never strongly ergodic. We give a short proof of Ueda's result. We refer to \cite{Co72} for the type classification of factors.

\begin{prop}[{\cite[Corollary 11]{Ue00}}]
Let $\mathcal R$ be any ergodic nonsingular equivalence relation defined on a standard probability space $(X, \mu)$. If $\mathcal R$ is of type ${\rm III_0}$, then $\mathcal R$ is not strongly ergodic.
\end{prop}

\begin{proof}
Put $A = \rL^\infty(X)$ and $M = \rL(\mathcal R)$. Assume that $\mathcal R$ is of type ${\rm III_0}$. Then $M$ is of type ${\rm III_0}$. Fix a nonprincipal ultrafilter $\omega$ on $\N$. Then $\mathcal Z(M^\omega) \neq \C 1$ by \cite[Theorem 6.18]{AH12}. Since $A^\omega$ is maximal abelian in $M^\omega$ by \cite[Theorem A.1.2]{Po95}, we have
$$\C 1 \neq \mathcal Z(M^\omega) = (M^\omega)' \cap M^\omega =(M^\omega)' \cap A^\omega \subset M' \cap A^\omega.$$
Since $M' \cap A^\omega \neq \C 1$, $\mathcal R$ is not strongly ergodic.
\end{proof}

\subsection*{Popa's intertwining-by-bimodules}

	In this subsection, we briefly recall Popa's intertwining-by-bimodules \cite{Po01, Po03}. In the present paper, we will need a generalization of Popa's intertwining-by-bimodules to the framework of type ${\rm III}$ von Neumann algebras developed by the authors in \cite{HI15}. We will use the following terminology (see \cite[Definition 4.1]{HI15}). 

\begin{df}\label{definition intertwining}\upshape
	Let $M$ be any $\sigma$-finite von Neumann algebra, $1_A$ and $1_B$ any nonzero projections in $M$, $A\subset 1_AM1_A$ and $B\subset 1_BM1_B$ any von Neumann subalgebras with faithful normal conditional expectations $\rE_A : 1_A M 1_A \to A$ and $\rE_B : 1_B M 1_B \to B$ respectively.  

	We say that $A$ {\em embeds with expectation into} $B$ {\em inside} $M$ and write $A \preceq_M B$ if there exist projections $e \in A$ and $f \in B$, a nonzero partial isometry $v \in eMf$ and a unital normal $\ast$-homomorphism $\theta : eAe \to fBf$ such that the inclusion $\theta(eAe) \subset fBf$ is with expectation and $av = v \theta(a)$ for all $a \in eAe$.
\end{df}

The main characterization of intertwining subalgebras we will use in this paper is the following result proven in \cite[Theorem 4.3]{HI15}.

\begin{thm}\label{intertwining for type III}
	Keep the same notation as in Definition \ref{definition intertwining} and moreover assume that $A$ is finite. 
Then the following conditions are equivalent.
	\begin{enumerate}
		\item $A \preceq_M B$.
	\item There exists no net $(w_i)_{i \in I}$ of unitaries in $\mathcal U(A)$ such that $\rE_{B}(b^*w_i a)\rightarrow 0$ in the $\sigma$-strong topology for all $a,b\in 1_AM1_B$.
\end{enumerate}
	\end{thm}

\section{Bi-exactness and Ozawa's condition (AO)}\label{AO}

\subsection*{Bi-exactness for discrete groups}
	Recall from \cite{Oz03} that a von Neumann algebra $M\subset \B(H)$ satisfies {\em condition} (AO) if there exist unital $\sigma$-weakly dense C$^{\ast}$-subalgebras $A\subset M$ and $B\subset M'$ such that $A$ is locally reflexive and the map 
\begin{equation*}
\nu : A\ota B\longrightarrow \B(H)/\K(H) : a\otimes b\mapsto ab+\K(H)
\end{equation*}
is continuous with respect to the minimal tensor norm. Recall that $A$ is {\em locally reflexive} or equivalently has {\em property} $C''$ (see e.g.\ \cite[Section 9]{BO08}) if for any C$^{\ast}$-algebra $C$, the inclusion map $A^{\ast\ast}\ota C\hookrightarrow (A\otm C)^{\ast\ast}$ is continuous with respect to the minimal tensor norm. In this case, any $\ast$-homomorphism $\pi\colon A\otm C \to \B(K)$ has an extension $\widetilde \pi : A^{**}\otm C  \to \B(K)$ which is normal on $A^{**}\otimes \C 1$ (since $\pi$ always has a canonical extension on $(A\otm C)^{**}$).

	We next recall the notion of {\em bi-exactness} for discrete groups which was introduced by Ozawa in \cite{Oz04} (using the terminology {\em class} $\mathcal S$) and intensively studied in \cite[Chapter 15]{BO08}. Our definition is different from the original one, but it is equivalent to it and it is moreover adapted to the framework of discrete quantum groups \cite[Definition 3.2]{Is13}.

\begin{df}[{\cite[Proposition 15.2.3(2)]{BO08}}]\upshape
	Let $\Gamma$ be any discrete group. We say that $\Gamma$ is {\em bi-exact} if there exists a $(\Gamma \times \Gamma)$-globally invariant unital $\rC^*$-subalgebra $\mathcal{B} \subset \ell^{\infty}(\Gamma)$ such that the following two conditions are satisfied:
\begin{itemize}
	\item[$\rm (i)$] The algebra $\mathcal{B}$ contains $c_0(\Gamma)$ so that the quotient $\mathcal{B}_\infty:=\mathcal{B}/c_0(\Gamma)$ is well-defined.
	\item[\rm (ii)] The left translation action $\Gamma \curvearrowright \ell^\infty(\Gamma)$ induces an amenable action $\Gamma \curvearrowright \mathcal{B}_\infty$ and the right translation action $\ell^\infty(\Gamma) \curvearrowleft \Gamma$ induces the trivial action on $\mathcal{B}_\infty$.
\end{itemize}
\end{df}

The class of bi-exact discrete groups includes amenable groups, free groups \cite{AO74}, discrete subgroups of simple connected Lie groups of real rank one \cite{Sk88} and Gromov word-hyperbolic groups \cite{Oz03}. Observe that for any bi-exact discrete group $\Gamma$, the group von Neumann algebra $\rL(\Gamma)$ satisfies condition (AO). We refer the reader to \cite[Chapter 15]{BO08} for more information on bi-exact discrete groups.

\subsection*{Ozawa's condition (AO) in crossed product von Neumann algebras}
In this subsection, we prove a relative version of Ozawa's condition (AO) in the framework of crossed product von Neumann algebras. This result will be used in the proof of Theorem \ref{thmA}. 

Let $\Gamma$ be any discrete group, $B \subset \mathcal B$ any inclusion of $\sigma$-finite von Neumann algebras and $\Gamma \curvearrowright \mathcal B$ any action that leaves globally invariant the subalgebra $B$. Denote by $M:= B\rtimes \Gamma$ and $\mathcal M = \mathcal B \rtimes \Gamma$ the corresponding crossed product von Neumann algebras, by $\rE_{\mathcal B} : \mathcal M \to \mathcal B$ the canonical faithful normal conditional expectation and by $e_{\mathcal B} : \rL^2(\mathcal M) \to \rL^2(\mathcal B)$ the corresponding Jones projection. We use the notation and terminology of Section \ref{preliminaries} for the standard forms $(\mathcal B, \rL^2(\mathcal B), J^{\mathcal B}, \mathfrak P^{\mathcal B})$ of $\mathcal B$ and $(\mathcal M, \rL^2(\mathcal M), J^{\mathcal M}, \mathfrak P^{\mathcal M})$ of $\mathcal M = \mathcal B \rtimes \Gamma$.

We define a nondegenerate (and possibly nonunital) $\rC^*$-algebra and its multiplier $\rC^*$-algebra inside $\B(\rL^2(\mathcal M))$ by
\begin{align*}
\mathcal{K_{\mathcal B}} &:= \mathrm{C}^* \left\{aJ^{\mathcal M}x J^{\mathcal M} e_{\mathcal B} b J^{\mathcal M}yJ^{\mathcal M}\mid a,b,x,y\in B\rtimes_{\rm red}\Gamma \right\} \subset \B(\rL^2(\mathcal M)) \\
\mathfrak{M}(\mathcal K_{\mathcal B})&:= \left\{ T \in \B(\rL^2(\mathcal M)) \mid T\mathcal K_{\mathcal B}\subset \mathcal K_{\mathcal B}  \text{ and } \mathcal K_{\mathcal B} T \subset \mathcal K_{\mathcal B} \right\}.
\end{align*}
where $\mathrm{C}^* \left\{ \mathcal Y \right \} \subset \mathbf B(\rL^2(\mathcal M))$ denotes the $\rC^*$-subalgebra of $\mathbf B(\rL^2(\mathcal M))$ generated by the subset $\mathcal Y \subset \mathbf B(\rL^2(\mathcal M))$. We record the following elementary lemma.

\begin{lem}\label{lemma AO3}
	We have
$$ \mathcal{K}_\mathcal{B} \subset \B(\rL^2(\mathcal{B})) \otm \K(\ell^2(\Gamma)).$$
\end{lem}
\begin{proof}
Denote by $\sigma : \Gamma \curvearrowright \mathcal B$ the action and by $u : \Gamma \to \mathcal U(\rL^2(\mathcal B))$ the canonical unitary representation implementing the action $\sigma$. Recall that $\rL^2(\mathcal M) = \rL^2(\mathcal B) \otimes \ell^2(\Gamma)$. Regard $\mathcal M = \mathcal B\rtimes \Gamma$ as generated by $\pi_\sigma(b) = \sum_{h \in \Gamma}\sigma_{h^{-1}}(b) \otimes P_{\C \delta_h}$ for $b \in \mathcal B$ and $1 \otimes \lambda_g$ for $g \in \Gamma$ where $P_{\C \delta_h} : \ell^2(\Gamma) \to \C \delta_h$ is the orthogonal projection onto $\C \delta_h$. We have $J^{\mathcal M} (1\otimes \lambda_g) J^{\mathcal M} = u_g \otimes \rho_g $ for all $g\in \Gamma$. Let $\mathcal C\subset \B(\rL^2(\mathcal B))$ be the $\rC^*$-algebra generated by $B$, $J^{\mathcal B}  B J^{\mathcal B}$ and $u_g$ for all $g\in \Gamma$. We will show that $\mathcal{K}=\mathcal C\otm \K(\ell^2(\Gamma))$.

	Recall that $e_{\mathcal B}=1\otimes P_{\C \delta_e}$. For all $g,h\in \Gamma$, denote by $e_{g,h} : \C \delta_h \to \C \delta_g$ the partial isometry sending $\delta_h$ onto $\delta_g$. For all $a,b\in B$ and all $g,h,s,t\in \Gamma$, we have
\begin{align*}
\pi_\sigma(a) (J^{\mathcal B} b J^{\mathcal B}\otimes 1) e_{\mathcal B} &= e_{\mathcal B} \pi_\sigma(a)(J^{\mathcal B} b J^{\mathcal B}\otimes 1) = aJ^{\mathcal B} b J^{\mathcal B} \otimes P_{\C\delta_e} \\
(1\otimes \lambda_g)  (u_s\otimes \rho_s) e_{\mathcal B} (1\otimes \lambda_h)(u_t\otimes  \rho_t) &= u_{st} \otimes \lambda_g\rho_{s}P_{\C\delta_e} \lambda_h\rho_{t}=u_{st}\otimes e_{gs^{-1}, h^{-1}t}.
\end{align*}
We then have
\begin{align*}
\mathcal K_{\mathcal B}&=\mathrm{C}^* \left\{aJ^{\mathcal M}x J^{\mathcal M} e_{\mathcal B} b J^{\mathcal M}yJ^{\mathcal M}\mid a,b,x,y\in B\rtimes_{\rm red}\Gamma \right\} \\
	&=\mathrm{C}^* \left\{aJ^{\mathcal B}b J^{\mathcal B} u_g \otimes e_{s,t} \mid a,b\in B , g,s,t\in \Gamma \right\}\\
	&=\mathcal C\otimes_{\rm min}\K(\ell^2(\Gamma)).
\end{align*}
This finishes the proof of Lemma \ref{lemma AO3}.
\end{proof}

Consider now the following unital $\ast$-homomorphism:
$$\nu_{\mathcal B} : (B\rtimes_{\rm red}\Gamma) \otimes_{\rm alg} J^{\mathcal M}(B\rtimes_{\rm red} \Gamma)J^{\mathcal M} \rightarrow \mathfrak{M}(\mathcal{K}_{\mathcal B})/\mathcal{K}_{\mathcal B} : a\otimes J^{\mathcal M}b J^{\mathcal M} \mapsto a \, J^{\mathcal M}bJ^{\mathcal M}+\mathcal{K}_{\mathcal B}.$$
Ozawa proved in \cite[Proposition 4.2]{Oz04} that when $\Gamma$ is bi-exact and $B = \mathcal B$ is finite and amenable, the map $\nu_{\mathcal B}$ is continuous with respect to the minimal tensor norm. This is nothing but a relative version of the condition (AO) in the framework of crossed product von Neumann algebras. Observe that when $B= \mathcal B = \C 1$, continuity of $\nu_{\mathcal B}$ with respect to the minimal tensor norm implies that $\rL(\Gamma)$ satisfies condition (AO). Since $\mathcal{K}_{\mathcal{B}}$ is the smallest $\rC^*$-algebra containing $1\otimes c_0(\Gamma)$ and such that its multiplier algebra contains $B\rtimes_{\rm red}\Gamma$ and $J^{\mathcal M}(B\rtimes_{\rm red} \Gamma)J^{\mathcal M}$, we can easily generalize \cite[Proposition 4.2]{Oz04} as follows. 
\begin{prop}\label{AO in crossed product}
	Keep the same setting as above and assume that $\Gamma$ is bi-exact and $B$ is amenable. Then the map 
$$\nu_{\mathcal B} : (B\rtimes_{\rm red}\Gamma) \otimes_{\rm alg} J^{\mathcal M}(B\rtimes_{\rm red} \Gamma)J^{\mathcal M} \rightarrow \mathfrak{M}(\mathcal{K}_{\mathcal B})/\mathcal{K}_{\mathcal B} : \ a\otimes J^{\mathcal M}b J^{\mathcal M} \mapsto a \, J^{\mathcal M}bJ^{\mathcal M}+\mathcal{K}_{\mathcal B}$$
is well-defined and continuous with respect to the minimal tensor norm.
\end{prop}
\begin{proof}
	As in the proof of Lemma \ref{lemma AO3}, regard $\mathcal M = \mathcal B \rtimes \Gamma$ as generated by $\pi_\sigma(\mathcal B)$ and $(1 \otimes \lambda)(\Gamma)$.  Since $B$ is amenable (i.e.\ semidiscrete), the map 
$$\pi_\sigma(B) \ota J^{\mathcal M} \pi_\sigma(B) J^{\mathcal M} \to \mathbf B(\rL^2(\mathcal M)) : \pi_\sigma(a)\otimes J^{\mathcal M} \pi_\sigma(b) J^{\mathcal M} \mapsto \pi_\sigma(a) \, J^{\mathcal M} \pi_\sigma(b) J^{\mathcal M}$$ 
is continuous with respect to the minimal tensor norm. Then the proof of \cite[Proposition 4.2]{Oz04} applies {\em mutatis mutandis} to show that the map $\nu_{\mathcal B}$ is continuous with respect to the minimal tensor norm.
\end{proof}
We will apply Proposition \ref{AO in crossed product} in Theorem \ref{theorem for thmA 1} (in the case when $\mathcal B = B$) and in Theorem \ref{theorem for thmA 2} (in the general case).

\subsection*{Ozawa's condition (AO) in ultraproduct von Neumann algebras}

In this subsection, we prove a version of Ozawa's condition (AO) in the framework of ultraproduct von Neumann algebras. Although we will not use this result in this paper, we nevertheless mention it since we believe it is interesting in its own right. 

We keep the same notation as in the previous subsection and we moreover assume that $B = \mathcal B$. Let $\omega \in \beta(\N) \setminus \N$ be any nonprincipal ultrafilter. Denote by $(M^\omega, \rL^2(M^\omega), J^{M^\omega}, \mathfrak P^{M^\omega})$ a standard form for $M^\omega$ and by $e_{B^\omega} : \rL^2(M^\omega) \to \rL^2(B^\omega)$ the Jones projection corresponding to the inclusion $B^\omega\subset M^\omega$. We define a (possibly degenerate and nonunital) C$^*$-subalgebra $\mathcal{K}_\omega$ and its multiplier algebra $\mathfrak{M}(\mathcal{K}_\omega)$ inside $\B(\rL^2(M^\omega))$ by
\begin{align*}
	\mathcal{K}_\omega &:= \mathrm{C}^* \left\{aJ^{M^\omega}xJ^{M^\omega} e_{B^\omega} b J^{M^\omega}yJ^{M^\omega} \mid a,b,x,y\in B\rtimes_{\rm red}\Gamma \right \} \subset \B(\rL^2(M^\omega)) \\
	\mathfrak{M}(\mathcal{K}_\omega) &:= \left \{ T \in \B(\rL^2(M^\omega))\mid T \mathcal{K}_\omega \subset \mathcal{K}_\omega \text{ and } \mathcal{K}_\omega T \subset \mathcal{K}_\omega \right\}.
\end{align*} 
Recall from Proposition \ref{AO in crossed product} (in the case when $B=\mathcal{B}$ with $\mathcal{K}:=\mathcal{K}_{\mathcal{B}}$ which is exactly \cite[Proposition 4.2]{Oz04}) that the map  
$$\nu \colon (B\rtimes_{\rm red}\Gamma) \otimes_{\rm alg} J^M(B\rtimes_{\rm red} \Gamma)J^M \rightarrow \mathfrak{M}(\mathcal{K})/\mathcal{K} : a\otimes J^M b J^M \mapsto a \, J^MbJ^M+\mathcal{K}$$
is continuous with respect to the minimal tensor norm. We now state a version of Ozawa's condition (AO) in the ultraproduct representation $\rL^2(M^\omega)$.

\begin{prop}\label{AO in ultraproduct}
	Keep the same setting as above and assume that $\Gamma$ is bi-exact and $B$ is amenable. Then the map
$$\nu_\omega : (B\rtimes_{\rm red}\Gamma) \otimes_{\rm alg} J^{M^\omega}(B\rtimes_{\rm red} \Gamma)J^{M^\omega} \rightarrow \mathfrak{M}(\mathcal{K}_\omega)/\mathcal{K}_\omega : a\otimes J^{M^\omega}bJ^{M^\omega} \mapsto a \, J^{M^\omega} bJ^{M^\omega}+\mathcal{K}_\omega$$
 is well-defined and continuous with respect to the minimal tensor norm.
\end{prop}
\begin{proof}
	Put
\begin{align*}
	C &:=\mathrm{C}^* \left\{ B\red\Gamma, J^M(B\red \Gamma)J^M \right\} \subset \B(\rL^2(M))\\
	C_\omega &:=\mathrm{C}^* \left\{ B\red\Gamma, J^{M^\omega}(B\red \Gamma)J^{M^\omega} \right\} \subset \B(\rL^2(M^\omega)).
\end{align*}
Observe that $C + \mathcal{K}$ (resp.\ $C_\omega + \mathcal{K}_\omega$) is a C$^*$-algebra since it is the sum of a $\rC^*$-subalgebra and an ideal in $\mathfrak{M}(\mathcal{K})$ (resp.\ $\mathfrak{M}(\mathcal{K}_\omega)$). 
\begin{claim}
There is a $\ast$-homomorphism $\theta :  C + \mathcal{K} \to \B(\rL^2(M^\omega))$ such that $\theta(x)=x$ and $\theta(J^M y J^M)= J^{M^\omega}y J^{M^\omega}$ for all $x,y\in B\red\Gamma$ and $\theta(e_B)=e_{B^\omega}$. 
\end{claim}

\begin{proof}[Proof of the Claim]
Indeed, fix any faithful state $\varphi \in M_\ast$ such that $\varphi \circ \rE_B = \varphi$ and denote by $p$ the support projection in $N =\prod^\omega M$ of the ultraproduct state $\varphi_\omega \in N_\ast$. By Lemma \ref{lemma for AO in ultraproduct}, $\pi^\omega((e_B)_n)$ commutes with $p$ and $J^N$ and hence $\pi^\omega((e_B)_n)$ commutes with $\widetilde{p}:= pJ^N p J^N$. Since $\widetilde{p}$ commutes with $\pi^\omega(M)$ and $\pi^\omega(J^M M J^M) = J^N \pi^\omega(M) J^N$, $\widetilde{p}$ commutes with $\pi^\omega(C + \mathcal{K})$. Recall that $\widetilde p N \widetilde p \cong p N p \cong M^\omega$ and $\widetilde p \rL^2(M)_\omega = \rL^2(M^\omega)$. Then the $\ast$-homomorphism
$$\theta : C + \mathcal{K} \to \B(\rL^2(M^\omega)) : T \mapsto \widetilde p \pi^\omega(T) \widetilde p$$ satisfies all the conditions of the Claim.
\end{proof}

Since $\theta(C)=C_\omega$ and $\theta(\mathcal{K}) = \mathcal{K}_\omega$, $\theta$ induces a $\ast$-homomorphism 
$$\widetilde{\theta} : (C + \mathcal{K})/\mathcal{K} \to (C_\omega + \mathcal{K}_\omega)/\mathcal{K}_\omega \subset \mathfrak{M}(\mathcal{K}_\omega)/\mathcal{K}_\omega.$$ 
Denote by $\iota : (B\rtimes_{\rm red}\Gamma) \otimes_{\min} J^{M^\omega}(B\rtimes_{\rm red} \Gamma)J^{M^\omega} \to (B\rtimes_{\rm red}\Gamma) \otimes_{\min} J^{M}(B\rtimes_{\rm red} \Gamma)J^{M}$ the tautological $\ast$-isomorphism. Then the composition map $$\nu_\omega = \widetilde{\theta}\circ \nu \circ \iota : (B\rtimes_{\rm red}\Gamma) \otimes_{\alg} J^{M^\omega}(B\rtimes_{\rm red} \Gamma)J^{M^\omega} \rightarrow \mathfrak{M}(\mathcal{K}_\omega)/\mathcal{K}_\omega$$
is continuous with respect to the minimal tensor norm. 
\end{proof}

\subsection*{Weak exactness for $\rC^*$-algebras}
To obtain structural results for von Neumann algebras $M$ satisfying Ozawa's (relative) condition (AO) \cite{Oz03, Oz04}, it is usually necessary to impose {\em local reflexivity} or {\em exactness} of the given unital $\sigma$-weakly dense $\rC^*$-algebra in $M$. Observe that in the setting of Proposition \ref{AO in crossed product}, the reduced crossed product $\rC^*$-algebra $B\red\Gamma$ need not be locally reflexive since it contains the von Neumann algebra $B$. To avoid this difficulty, Ozawa assumed in \cite[Theorem 4.6]{Oz04} that $\Gamma$ is exact and $B$ is {\em abelian} so that $B\red\Gamma$ is exact and hence locally reflexive. In \cite{Is12}, the second named author introduced a notion of {\em weak exactness} for $\rC^*$-algebras and could settle this problem. Namely, he generalized \cite[Theorem 4.6]{Oz04} under the assumptions that $\Gamma$ is exact and $B$ is  {\em amenable} (not necessarily abelian). The main idea behind this generalization was to use {\em some} exactness (or equivalently property $C'$) of the opposite algebra $(B\red \Gamma)^{\rm op}$, instead of local reflexivity of $B\red \Gamma$. In the present paper, to study more general cases, we will make use of this notion of weak exactness for $\rC^*$-algebras.

	Recall from \cite[Theorem 3.1.3(1)(ii)]{Is12} that for an inclusion of a unital $\rC^*$-algebra $A\subset M$ in a von Neumann algebra $M$, we say that \textit{$A$ is weakly exact in $M$} if for any unital $\rC^*$-algebra $C$, any $\ast$-homomorphism $\pi\colon A\otm C \to \B(K)$ which is $\sigma$-weakly continuous on $A\otimes  \C1$ has an extension $\widetilde \pi : A\otm C^{**} \to \B(K)$ which is normal on $\C1 \otimes C^{**}$. In the case when $A=M$, we simply say that $M$ is weakly exact. Here we recall the following fundamental fact.

\begin{prop}[{\cite[Proposition 4.1.7]{Is12}}]\label{weakly exact for crossed product}
Let $\Gamma$ be any exact discrete group, $B$ any $\sigma$-finite amenable (and hence weakly exact) von Neumann algebra and $\Gamma \curvearrowright B$ any action. Then the reduced crossed product $\rC^*$-algebra $B\red\Gamma$ is weakly exact in $B \rtimes \Gamma$. If moreover $\Gamma$ is countable and $B$ has separable predual, then $B\rtimes \Gamma$ is weakly exact.
\end{prop}

	Using this property, we prove an important lemma, which is a variant of {\cite[Lemma 5]{Oz03}} (see also \cite[Proposition 15.1.6]{BO08} and \cite[Lemma 5.1.1]{Is12}). The proof is essentially the same as the one of \cite[Proposition 15.1.6]{BO08} but we nevertheless include it for the reader's convenience.

\begin{lem}\label{lemma for AO}
	Let $M\subset \mathcal{M}$ be any inclusion of $\sigma$-finite von Neumann algebras with expectation and $(\mathcal{M}, \rL^2(\mathcal{M}), J^{\mathcal M}, \mathfrak P^{\mathcal M})$ a standard form for $\mathcal M$.  Let $C\subset M$ be any unital $\sigma$-weakly dense $\rC^*$-subalgebra, $p\in \mathcal{M}$ any nonzero projection and $\varphi : M \to p\mathcal{M}p$ any normal ucp map. We will use the identification $p \mathbf B(\rL^2(\mathcal M))p = \mathbf B(p \rL^2(\mathcal M))$.
	
Assume that the following two conditions hold:
\begin{itemize}
	\item The map 
$$\Phi : C \ota J^{\mathcal M}CJ^{\mathcal M} \to \B(p\rL^2(\mathcal{M})) : \sum_{i=1}^n x_i \otimes J^{\mathcal M}y_iJ^{\mathcal M} \mapsto \sum_{i=1}^n \varphi(x_i) \, J^{\mathcal M}y_iJ^{\mathcal M}p$$
is continuous with respect to the minimal tensor norm.
	\item The $\rC^*$-algebra $C$ is locally reflexive or $C$ is weakly exact in $M$.
\end{itemize}
Then the ucp map $\varphi : M \to p\mathcal{M}p$ has a ucp extension $\widetilde{\varphi}: \B(\rL^2(\mathcal{M})) \to(J^{\mathcal M}CJ^{\mathcal M}p)'\cap \B(p\rL^2(\mathcal{M}))$.
\end{lem}

\begin{proof}
To simplify the notation, we will write $J = J^{\mathcal M}$. Observe that $\Phi = \nu \circ (\varphi \otimes \id_{JCJ})$ where $\nu : p\mathcal{M}p \ota JCJ \to \mathbf B(p \rL^2(\mathcal M))$ is the multiplication map. We first prove the following result.

\begin{claim}
The ucp map $\Phi : C \otimes_{\min} JCJ \to \B(p\rL^2(\mathcal{M}))$ can be extended to a ucp map $ \widetilde \Phi : M \otimes_{\min} JCJ \to \B(p\rL^2(\mathcal{M}))$ which is normal $M \otimes \C 1$. In particular, we have $\widetilde \Phi(x \otimes 1) = \varphi(x)$ for all $x \in M$.
\end{claim}

\begin{proof}[Proof of the Claim]
Indeed, let $(\pi, V, K)$ be a {\em minimal} Stinespring dilation for $\Phi : C \otimes_{\min} JCJ \to \B(p\rL^2(\mathcal{M}))$, that is, $\pi : C\otm JCJ\to \B(K)$ is a unital $\ast$-representation and $V : p\rL^2(\mathcal{M})\to K$ is an isometry such that the subspace $\pi(C \otimes_{\min} JCJ)V p \rL^2(\mathcal M)$ is dense in $K$ and $\Phi(x) = V^* \pi(x) V$ for all $x \in C$. By minimality of $(\pi, V, K)$ and since $\Phi$ is $\sigma$-weakly continuous on $C \otimes \C 1$ (resp.\ $\C1 \otimes JCJ$), we have that $\pi$ is also $\sigma$-weakly continuous on $C \otimes \C 1$ (resp.\ $\C 1 \otimes JCJ$). Indeed, it suffices to notice that for all $c_1, c_2 \in C$, all $x, y \in C \otimes_{\min} JCJ$ and all $\xi , \eta \in p \rL^2(\mathcal M)$, we have
\begin{align*}
\langle \pi (c_1 \otimes Jc_2J) \, \pi(x)V \xi, \pi(y)V\eta\rangle_K &= \langle V^* \pi(y^*(c_1 \otimes Jc_2J)x) V \xi , \eta\rangle_K \\
&= \langle \Phi(y^*(c_1 \otimes Jc_2J)x) \xi, \eta\rangle_K.
\end{align*}
 Since $C$ is assumed to be locally reflexive or weakly exact in $M$ (which is equivalent to saying that $JCJ$ is weakly exact in $JMJ$), the unital $\ast$-homomorphism $\pi : C\otm JCJ\to \B(K)$ always has an extension $\widetilde{\pi} : C^{**}\otm JCJ \to \mathbf B(K)$ which is normal on $C^{**} \otimes \C 1$. Observe that we do not need $\sigma$-weak continuity on $\C1\otimes JCJ$ when $C$ is locally reflexive.

Let $z\in C^{**}$ be the central projection such that $zC^{**} = M$ canonically and let $z_i\in C$ be a bounded net converging to $z$ in the $\sigma$-weak topology in $C^{**}$. Observe that $z_i \to 1_M$ $\sigma$-weakly in $M$ and recall that $\widetilde{\pi}$ (resp.\ $\pi$) is $\sigma$-weakly continuous on $C^{**}\otimes \C1 $ (resp.\ $C\otimes \C1 \subset M\otimes \C1$). 
We then have that
$$\widetilde{\pi}(z\otimes 1) = \lim_i \pi(z_i \otimes 1) = \pi(1_M \otimes 1)=1$$
and hence $\widetilde{\pi}((1-z)\otimes 1)=0$. Since $\widetilde{\pi}$ is a $\ast$-homomorphism, it satisfies $\widetilde{\pi}((z\otimes 1) x )= \widetilde{\pi}(x)$ for all $x\in C^{**}\otm JCJ$. In particular, we have $\widetilde \pi ((z \otimes 1) \, \cdot \,)|_{C \otimes_{\min} JCJ} = \pi$. Using moreover the identification $M \otimes_{\min} JCJ = z C^{**} \otimes_{\min} JCJ$, we obtain that the unital $\ast$-homomorphism $\widetilde \pi((z \otimes 1) \, \cdot \, ) : M \otimes_{\min} JCJ \to \mathbf B(K)$ is an extension of $\pi : C \otimes_{\min} JCJ \to \mathbf B(K)$ which is normal on $M \otimes \C1$. Therefore the ucp map $\widetilde \Phi = \Ad(V^*) \circ \widetilde \pi((z \otimes 1) \, \cdot \, ) : M \otimes_{\min} JCJ \to \mathbf B(p \rL^2(\mathcal M))$ is an extension of $\Phi : C \otimes_{\min} JCJ \to \B(p\rL^2(\mathcal{M}))$ which is normal on $M \otimes \C1$. In particular, we have $\widetilde \Phi(x \otimes 1) = \varphi(x)$ for all $x \in M$.
\end{proof}

We next apply Arveson's extension theorem to the ucp map $\widetilde \Phi : M \otimes_{\min} JCJ \to \B(p\rL^2(\mathcal{M}))$ and we obtain a ucp extension map that we still denote by $\widetilde \Phi : \B(\rL^2(\mathcal{M}))\otm JCJ \to \B(p\rL^2(\mathcal{M}))$. Since $\widetilde \Phi |_{\C1\otimes JCJ} : \C1 \otimes JCJ \to \mathbf B(p \rL^2(\mathcal M)) : 1 \otimes JxJ \mapsto JxJp$ is a unital $\ast$-homomorphism, $\C1\otimes JCJ$ is contained in the multiplicative domain of $\widetilde \Phi$ (see e.g.\ \cite[Section 1.5]{BO08}). Therefore, for all $u\in \mathcal{U}( C)$ and all $x\in \B(\rL^2(\mathcal{M}))$, we have 
$$\Phi(x\otimes 1) \, JuJ p = \Phi(x\otimes 1)\Phi(1\otimes JuJ)= \Phi(x\otimes  JuJ)=\Phi(1\otimes JuJ)\Phi(x\otimes 1)=JuJp \, \Phi(x\otimes 1)$$
and hence $\Phi(\B(\rL^2(\mathcal{M}))\otimes 1)\subset (JCJp)' \cap \B(p\rL^2(\mathcal{M}))$. Thus, $\widetilde{\varphi}:=\Phi(\, \cdot \, \otimes 1) : \B(\rL^2(\mathcal{M})) \to(JCJp)'\cap \B(p\rL^2(\mathcal{M}))$ is the desired ucp extension map.
\end{proof}

\section{Proofs of Theorem \ref{thmA} and Corollary \ref{corB}}\label{section-thmA}

We first prove two intermediate results, namely Theorems \ref{theorem for thmA 1} and \ref{theorem for thmA 2}, from which we will deduce Theorem \ref{thmA}. While these two results are independent from each other, their proofs are in fact very similar and use Ozawa's condition (AO) for crossed product von Neumann algebras from Section \ref{AO}. 

Theorem \ref{theorem for thmA 1} below is a spectral gap rigidity result for subalgebras with expectation $N \subset M$ of crossed product von Neumann algebras $M = B \rtimes \Gamma$ arising from arbitrary actions of bi-exact discrete groups on amenable von Neumann algebras.

\begin{thm}\label{theorem for thmA 1}
	Let $\Gamma$ be any bi-exact discrete group, $B$ any amenable $\sigma$-finite von Neumann algebra and $\Gamma \curvearrowright B$ any action. Put $M:=B\rtimes \Gamma$. Let $p\in M$ be any nonzero projection and $N\subset pMp$ any von Neumann subalgebra with expectation. Let $\omega \in \beta(\N) \setminus \N$ be any nonprincipal ultrafilter.

Then at least one of the following conditions holds true.
\begin{itemize}
	\item The von Neumann algebra $N$ has a nonzero amenable direct summand.	
	\item We have $N'\cap pM^\omega p \subset p (B^\omega \rtimes \Gamma) p$.
\end{itemize}
\end{thm}
\begin{proof}
	Assume that $N' \cap pM^\omega p \not \subset p(B^\omega \rtimes \Gamma)p$. Let $Y \in N' \cap pM^\omega p$ be such that $Y \notin p(B^\omega \rtimes \Gamma)p$. Up to replacing $Y$ by $Y - \rE_{B^\omega \rtimes \Gamma}(Y) \neq 0$ which still lies in $N' \cap pM^\omega p$, we may assume that $Y \in N' \cap p M^\omega p$, $Y \neq 0$ and $\rE_{B^\omega \rtimes \Gamma}(Y) = 0$. Put $y = \rE_\omega(Y^*Y) \in (N' \cap pMp)^+$. Define the nonzero spectral projection $p_0 := \mathbf 1_{[\frac12\|y\|_\infty, \|y\|_\infty]}(y) \in N' \cap pMp$ and put $c := (yp_0)^{-1/2} \in N' \cap pMp$. We have 
$$\rE_\omega ((Yc)^*(Yc)) = \rE_\omega(c \, Y^*Y \, c) = c \, \rE_{\omega}(Y^*Y) \, c = c \, y \, c = p_0.$$
Up to replacing $Y$ by $Yc$ which still lies in $N' \cap pM^\omega p$, we may assume that $Y \in N' \cap p M^\omega p$, $Y \neq 0$, $\rE_{B^\omega \rtimes \Gamma}(Y) = 0$ and $\rE_\omega(Y^*Y) = p_0$. Write $Y = (y_n)^\omega$ for some $(y_n)_n \in \mathcal M^\omega(M)$. Observe that $\sigma\text{-weak} \lim_{n \to \omega} y_n^*  y_n = \rE_\omega(Y^*Y) = p_0$.

Denote by $(M, \rL^2(M), J^M, \mathfrak P^M)$ a standard form for $M = B \rtimes \Gamma$ as in Section \ref{preliminaries}. To further simplify the notation, we will write $J = J^M$ and $\mathfrak P = \mathfrak P^M$. Define the cp map 
$$\Psi : \mathbf B(\rL^2(M)) \to \mathbf B(\rL^2(M)) : T \mapsto \sigma\text{-weak} \lim_{n \to \omega} y_n^* T y_n.$$
Observe that $\Psi(1) = p_0$. Since $\Psi$ is a cp map and $\Psi(1) = p_0$ is a projection, we have $\Psi(T) = \Psi(1) \Psi(T) \Psi(1) = p_0 \Psi(T) p_0$ for every $T \in \mathbf B(\rL^2(M))$ and hence $\Psi(\mathbf B(\rL^2(M))) \subset \mathbf B(p_0\rL^2(M))$ using the identification $p_0 \mathbf B(\rL^2(M)) p_0 = \mathbf B(p_0 \rL^2(M))$. We will then regard $\Psi : \mathbf B(\rL^2(M)) \to \mathbf B(p_0\rL^2(M))$ as a ucp map. Observe that $\Psi(x) = \rE_\omega(Y^*x Y)$ for all $x \in M$ and hence $\Psi(M) \subset p_0Mp_0$ and $\Psi |_M$ is normal. Moreover, observe that $\Psi |_N : N \to \mathbf B(p_0\rL^2(M)) : x \mapsto xp_0$ and $\Psi |_{JMJ} : {JMJ} \to \mathbf B(p_0\rL^2(M)) : JxJ \mapsto JxJ p_0$ are unital $\ast$-homomorphisms. We will denote by $\psi := \Psi |_M : M \to p_0Mp_0 : x \mapsto \Psi(x)$. 

Let $\mathcal K_B$ as in Proposition \ref{AO in crossed product} (for $\mathcal B = B$). For all $a, b \in M$, we have 
$$\rE_{B^\omega}(b^* Y a) = \rE_{B^\omega}( \rE_{B^\omega \rtimes \Gamma}(b^* Y a)) = \rE_{B^\omega}(b^* \rE_{B^\omega \rtimes \Gamma}( Y) a) = 0.$$
Since $0 = \rE_{B^\omega}(b^* Y a) = (\rE_B(b^* y_n a))^\omega$, we obtain that $\rE_B(b^* y_n a) \to 0$ $\sigma$-strongly as $n \to \omega$. Choose any cyclic unit vector $\xi \in \mathfrak P$ such that $e_B \xi = \xi$. For all $a, b, c, d \in M$, we have
\begin{align*}
\left|\langle \Psi(a e_B b) \, c\xi , d\xi \rangle_{\rL^2(M)} \right| &= \lim_{n \to \omega} \left|\langle y_n^* a e_B b y_n \, c \xi, d\xi \rangle_{\rL^2(M)} \right| \\
&= \lim_{n \to \omega} \left| \langle   e_B b y_n c \xi, e_B a^*y_nd\xi \rangle_{\rL^2(M)} \right| \\
&= \lim_{n \to \omega}\left| \langle   e_B \, b y_n c \, e_B \xi, e_B \, a^*y_nd \, e_B\xi \rangle_{\rL^2(M)} \right| \\
&= \lim_{n \to \omega} \left| \langle   \rE_B( b y_n c) e_B\xi, \rE_B( a^*y_nd) e_B\xi \rangle_{\rL^2(M)} \right| \\
&= \lim_{n \to \omega} \left| \langle   \rE_B( b y_n c) \xi, \rE_B( a^*y_nd) \xi \rangle_{\rL^2(M)} \right| \\
& \leq \lim_{n \to \omega} \| \rE_B( b y_n c) \xi \|_{\rL^2(M)} \|\rE_B( a^*y_nd) \xi\|_{\rL^2(M)} \\
&= 0.
\end{align*}
This implies that $\Psi(a e_B b) = 0$. By construction, we have $M = B \rtimes \Gamma$ and hence $e_B$ corresponds to the projection $1 \otimes P_{\C \delta_e}$. Taking $a = \lambda_g$ and $b = \lambda_h$ for $g, h \in \Gamma$, we then obtain $\Psi \left(\C 1 \otimes \mathbf K(\ell^2(\Gamma)) \right) = 0$. Since $\Psi$ is a ucp map, we obtain $\Psi \left(\mathbf B(\rL^2(B)) \otimes_{\min} \mathbf K(\ell^2(\Gamma)) \right) = 0$ and hence $\Psi(\mathcal K_B) = 0$ using Lemma \ref{lemma AO3}. Define the ucp map $\overline \Psi : \mathfrak M(\mathcal{K}_B)/\mathcal{K}_B \to \mathbf B(p_0\rL^2(M)) : a + \mathcal K_B \mapsto \Psi(a)$.

Using Proposition \ref{AO in crossed product} in the case when $B = \mathcal B$, we may then define the ucp composition map
\begin{equation*}
\Phi = \overline \Psi \circ \nu : (B \red \Gamma) \otimes_{\min} J(B \red \Gamma)J \to \mathbf B(p_0\rL^2(M)) : a \otimes JbJ \mapsto \Psi(a \, JbJ).
\end{equation*}
Since $\Psi |_{JMJ}$ is a unital $\ast$-homomorphism  and since  $\psi = \Psi |_M$ by definition, we have $\Phi(a \otimes JbJ) = \Psi(a \, JbJ) = \Psi(a) \, \Psi(JbJ) = \psi(a) \, JbJp_0$ for all $a, b \in B \red \Gamma$. Since 
$$(J(B \red \Gamma)J p_0)' \cap \mathbf B(p_0\rL^2(M)) = p_0 (JMJ)' p_0 \cap \mathbf B(p_0\rL^2(M)) = p_0Mp_0,$$ Proposition \ref{weakly exact for crossed product} and Lemma \ref{lemma for AO} imply that the normal ucp map $\psi : M \to p_0Mp_0$ has a ucp extension $\widetilde \psi : \mathbf B(\rL^2(M)) \to p_0Mp_0$. Observe that $Np_0 \subset p_0 M p_0$ is still with expectation by \cite[Proposition 2.2]{HU15}. Denote by $\rE_{Np_0} : p_0 Mp_0 \to Np_0$ a faithful normal conditional expectation. Define the unital $\ast$-homomorphism $\iota : N \to Np_0 : x \mapsto xp_0$ and denote by $z \in \mathcal Z(N)$ the unique central projection such that $\ker(\iota) = N z^\perp$. Then $Nz \cong Np_0$ and $Nzp_0 = Np_0$. Define the ucp map $\Theta =  \iota^{-1} \circ \rE_{Np_0} \circ \widetilde \psi (z \cdot z) : \mathbf B(z\rL^2(M)) \to Nz$. Since $\Theta |_{Nz} = \id_{Nz}$, $\Theta$ is a norm one projection and hence $Nz$ is amenable. We have therefore proved that if $N' \cap pM^\omega p \not\subset p(B^\omega \rtimes \Gamma)p$, then $N$ has a nonzero amenable direct summand.
\end{proof}

\begin{thm}\label{theorem for thmA 2}
	Let $\Gamma$ be any bi-exact discrete group, $B\subset \mathcal{B}$ any inclusion of $\sigma$-finite von Neumann algebras with expectation and $\Gamma \curvearrowright \mathcal B$ any action that leaves the subalgebra $B$ globally invariant. Assume moreover that $B$ is amenable. Put $M:=B\rtimes \Gamma\subset \mathcal{B}\rtimes \Gamma=:\mathcal{M}$. Let $p\in M$ be any nonzero projection and $N\subset pMp$ any von Neumann subalgebra with expectation. 

Then at least one of the following conditions holds true.
\begin{itemize}
	\item The von Neumann algebra $N$ is amenable.	
	\item We have $A\preceq_{\mathcal{M}} \mathcal{B}$ for any finite von Neumann subalgebra $A\subset N'\cap p\mathcal{M} p$ with expectation.
\end{itemize}
\end{thm}
\begin{proof}
Since the proof is very similar to the one of Theorem \ref{theorem for thmA 1}, we will simply sketch it and point out the necessary changes compared to Theorem \ref{theorem for thmA 1}. Denote by $(\mathcal M, \rL^2(\mathcal M), J^{\mathcal M}, \mathfrak P^{\mathcal M})$ a standard form for $\mathcal M = \mathcal B \rtimes \Gamma$ as in Section \ref{preliminaries}. Suppose that there exists a finite von Neumann subalgebra $A\subset N'\cap p\mathcal{M}p$ with expectation such that $A\not\preceq_{\mathcal{M}} \mathcal{B}$. Observe that since $N'\cap p\mathcal{M}p \subset p\mathcal Mp$ is with expectation, so is $A\subset p\mathcal{M}p$. We will show $N$ is amenable. We will use the identifications $p\mathbf B(\rL^2(M))p = \mathbf B(p \rL^2(M))$ and $p\mathbf B(\rL^2(\mathcal M))p = \mathbf B(p \rL^2(\mathcal M))$.

	Take a net of unitaries $(u_i)_{i\in I}$ in $\mathcal{U}(A)$ as in Theorem \ref{intertwining for type III}($\rm ii$) such that $\rE_\mathcal{B}(b^* u_i a) \to 0$ $\sigma$-strongly for any $a,b \in \mathcal{M}$. 
Fix a cofinal ultrafilter $\mathcal{U}$ on the directed set $I$ and define the ucp map
$$\Psi : \B(\rL^2(\mathcal{M})) \to \B(p\rL^2(\mathcal{M})): T \mapsto \sigma\text{-weak } \lim_{i\to \mathcal{U}} u_i^* T u_i.$$
Observe that $\Psi|_{\mathcal M} : \mathcal{M}\to p\mathcal{M}p$ is normal. Indeed, since $A \subset p\mathcal{M}p$ is finite and with expectation and since $\mathcal M$ is $\sigma$-finite, there exists a faithful state $\varphi \in (p \mathcal M p)_\ast$ such that $A\subset (p \mathcal{M} p)^\varphi$. Since $u_i \in \mathcal U(A)$ for all $i \in I$, this implies that $(\varphi \circ \Psi)(p  x  p) = \varphi(x)$ for all $x \in p \mathcal M p$. Since $\varphi$ is faithful and normal and since $\Psi = \Psi (p \cdot p)$, it follows that $\Psi|_{\mathcal M} : \mathcal M \to p\mathcal Mp$ is indeed normal. Moreover, we have $\Psi (x) = x$ for all $x \in N$.

	Let $\mathcal{K}_{\mathcal{B}}$ be as in Proposition \ref{AO in crossed product}. By a reasoning entirely similar to the one of the proof of Theorem \ref{theorem for thmA 1}, we have $\Psi(\mathcal{K}_{\mathcal{B}}) = 0$. Define the ucp map $\overline \Psi : \mathfrak M(\mathcal K_{\mathcal B})/\mathcal K_{\mathcal B} \to \mathbf B(p\rL^2(\mathcal M)) : a + \mathcal K_{\mathcal B} \mapsto \Psi(a)$. Using Proposition \ref{AO in crossed product}, we may then define the ucp composition map
\begin{equation*}
\overline \Psi \circ \nu_{\mathcal B} : (B \red \Gamma) \otimes_{\min} J^{\mathcal M}(B \red \Gamma)J^{\mathcal M} \to \mathbf B(p\rL^2(\mathcal M)) : a \otimes J^{\mathcal M}bJ^{\mathcal M} \mapsto \Psi(a \, J^{\mathcal M}bJ^{\mathcal M}).
\end{equation*}
Proposition \ref{weakly exact for crossed product} and Lemma \ref{lemma for AO} imply that the normal ucp map $\psi := \Psi |_M : M \to p\mathcal Mp$ has a ucp extension $\widetilde \psi : \mathbf B(\rL^2(\mathcal M)) \to (J^{\mathcal M}(B\red \Gamma)J^{\mathcal M}p)'\cap  \B(p\rL^2(\mathcal{M}))$. Denote by $e_M : \rL^2(\mathcal M) \to \rL^2(M)$ the Jones projection corresponding to the inclusion $M \subset \mathcal M$. We then have the identifications $e_M \B(p\rL^2(\mathcal M)) e_M = \B(p\rL^2(M))$ and $e_M J^{\mathcal M} e_M = J^M$. Then we have 
$$e_M \left( (J^{\mathcal M}(B\red \Gamma)J^{\mathcal M}p)'\cap  \B(p\rL^2(\mathcal{M})) \right) e_M = (J^M(B\red \Gamma)J^Mp)'\cap  \B(p\rL^2({M})) = pMp$$
and hence the ucp map $\widetilde \Psi := e_M \, \widetilde \psi(\, \cdot \,) \, e_M : \B(\rL^2(\mathcal M)) \to pMp$ takes indeed values in $pMp$. Moreover, we have $\widetilde \Psi(x) = x$ for all $x \in N$. If we denote by $\rE_N : p M p \to N$ a faithful normal conditional expectation, the ucp map $\Theta = \rE_N \circ \widetilde \Psi(p \cdot p) : \mathbf B(p \rL^2(\mathcal M)) \to N$ is a norm one projection. Therefore, $N$ is amenable.
\end{proof}

\begin{proof}[Proof of Theorem \ref{thmA}]
	Suppose that $N$ has no amenable direct summand. Then by Theorem \ref{theorem for thmA 1}, we have $N'\cap pM^\omega p \subset p(B^\omega \rtimes \Gamma) p$. We then apply Theorem \ref{theorem for thmA 2} to $N$ in the case when $\mathcal{B}:=B^\omega$ and we obtain $A\preceq_{B^\omega \rtimes \Gamma} B^\omega$ for any finite von Neumann subalgebra $A \subset N'\cap pM^\omega p$ with expectation. 
\end{proof}

\begin{proof}[Proof of Corollary \ref{corB}]
	Let $N\subset M$ be any von Neumann subalgebra with expectation such that $N' \cap M^\omega$ has no type ${\rm I}$ direct summand. Denote by $p \in \mathcal Z(N)$ the unique central projection such that $Np$ has no amenable direct summand and $N(1 - p)$ is amenable. Assume by contradiction that $p \neq 0$. Then $(Np)' \cap p M^\omega p \subset p(B^\omega \rtimes \Gamma)p$ by Theorem \ref{thmA} and $(Np)' \cap p M^\omega p = p(N' \cap M^\omega)p$ has no type ${\rm I}$ direct summand. By \cite[Corollary 8]{CS78} (see also \cite[Theorem 11.1]{HS90}), $(Np)'\cap pM^\omega p$ contains a copy of the hyperfinite $\rm II_1$ factor $R$ with expectation. We then have $R\preceq_{B^\omega \rtimes \Gamma} B^\omega$ by Theorem \ref{thmA}. Since $R$ is of type ${\rm II_1}$ and $B^\omega$ is abelian and hence of type ${\rm I}$, we obtain a contradiction. Therefore, $p = 0$ and $N$ is amenable.
\end{proof}

\section{Proof of Theorem \ref{thmC}}

We start by proving a useful lemma which can be regarded as a generalization of the first part of the proof of \cite[Proposition C]{Ho15}.

\begin{lem}\label{lem-strong-ergodicity}
Let $\mathcal R$ be any strongly ergodic nonsingular equivalence relation defined on a standard probability space $(X, \mu)$. Put $A = \rL^\infty(X)$ and $M = \rL(\mathcal R)$. Denote by $\rE_A : M \to A$ the unique faithful normal conditional expectation. Fix any faithful state $\tau \in A_\ast$ and put $\varphi = \tau \circ \rE_A \in M_\ast$.

If $M$ is not full, then there exists a sequence of unitaries $u_n \in \mathcal U(M)$ such that the following conditions hold:
\begin{itemize}
\item[$\rm(i)$]  $\lim_n \|u_n \varphi - \varphi u_n\| = 0$,
\item[$\rm(ii)$]  $\lim_n \|x u_n - u_n x\|_\varphi = 0$ for all $x \in M$ and
\item[$\rm(iii)$]  $\lim_n \|\rE_A(x u_n y)\|_\varphi = 0$ for all $x, y \in M$.
\end{itemize}
\end{lem}

\begin{proof}
Assume that $M$ is not full. Then $M' \cap (M^\omega)^{\varphi^\omega}$ is diffuse by \cite[Corollary 2.6]{HR14} for any nonprincipal ultrafilter $\omega \in \beta(\N) \setminus \N$. Then a combination of the first part of the proof of \cite[Theorem A]{HR14} and Lemma \ref{lem-ultraproduct} shows that there exists a sequence of unitaries $u_n \in \mathcal U(M)$ such that the following conditions hold:
\begin{itemize}
\item[$\rm(i)$]  $\lim_n \|u_n \varphi - \varphi u_n\| = 0$,
\item[$\rm(ii)$]  $\lim_n \|x u_n - u_n x\|_\varphi = 0$ for all $x \in M$ and
\item[$\rm(iii)$]  $u_n \to 0$ $\sigma$-weakly as $n \to \infty$.
\end{itemize}
It remains to prove that Conditions $\rm(i), \rm(ii), \rm(iii)$ imply that $\lim_n \|\rE_A(x u_n y)\|_\varphi = 0$ for all $x, y \in M$. The rest of the proof is entirely analogous to the first part of the proof of \cite[Proposition C]{Ho15} and we only give the details for the sake of completeness. Observe that for every nonprincipal ultrafilter $\omega \in \beta(\N) \setminus \N$, Condition $\rm(i)$ implies that $(u_n)_n \in \mathcal M^\omega(M)$ and Conditions $\rm(ii)$ and $\rm(iii)$ imply that $(u_n)^\omega \in M' \cap (M^\omega)^{\varphi^\omega}$ and $\varphi^\omega((u_n)^\omega) = 0$. We start by proving the following claim.

\begin{claim}
We have $\lim_n \|\rE_A(u_n)\|_\varphi = 0$
\end{claim}

\begin{proof}[Proof of the Claim]
Let $g \in [\mathcal R]$ be any element and denote by $u_g \in \mathcal U(\rL(\mathcal R))$ the corresponding unitary element. Since $u_g \rE_A(u_n) u_g^* = \rE_A(u_g u_n u_g^*)$, we have
\begin{align*}
\|\rE_A(u_n) u_g^* - u_g^* \rE_A(u_n)\|_\varphi &= \|u_g \rE_A(u_n) u_g^* - \rE_A(u_n)\|_\varphi \\
&= \|\rE_A(u_g u_n u_g^* - u_n)\|_\varphi \\
&\leq \|u_g u_n u_g^* - u_n\|_\varphi \\
&= \| u_n u_g^* - u_g^* u_n\|_\varphi \to 0 \quad \text{as} \quad n \to \infty.
\end{align*}
Define $\mathcal E= \spn \left\{a u_g \mid a \in A, g \in [\mathcal R] \right\}$ and observe that $\mathcal E$ is a unital $\sigma$-strongly dense $\ast$-subalgebra of $M$. The above calculation implies that $\lim_n \|x \rE_A(u_n) - \rE_A(u_n) x\|_\varphi = 0$ for every $x \in \mathcal E$. Let $\omega \in \beta(\N) \setminus \N$ be any nonprincipal ultrafilter. For every $x \in \mathcal E$, we have $\|x \rE_{A^\omega}((u_n)^\omega) - \rE_{A^\omega}((u_n)^\omega) x\|_{\varphi^\omega} = \lim_{n \to \omega} \|x \rE_A(u_n) - \rE_A(u_n) x\|_\varphi = 0$ and hence we have $x \rE_{A^\omega}((u_n)^\omega) = \rE_{A^\omega}((u_n)^\omega) x$. Since $\mathcal E$ is $\sigma$-strongly dense in $M$, this further implies that $x \rE_{A^\omega}((u_n)^\omega) = \rE_{A^\omega}((u_n)^\omega) x$ for every $x \in M$. Since $\mathcal R$ is strongly ergodic, this implies that $\rE_{A^\omega}((u_n)^\omega) = \varphi^\omega ((u_n)^\omega) 1 = 0$ and hence $\lim_{n \to \omega} \|\rE_A(u_n)\|_\varphi = \|\rE_{A^\omega}((u_n)^\omega)\|_{\varphi^\omega} = 0$. Since this is true for every $\omega \in \beta(\N) \setminus \N$, we finally obtain that $\lim_n \|\rE_A(u_n)\|_\varphi = 0$.
\end{proof}

We can now finish the proof of Lemma \ref{lem-strong-ergodicity}. Let $g \in [\mathcal R]$ be any element such that $g^2 = 1$. Put $X_g = \{s \in X \mid g \cdot s = s\}$ and observe that $u_g = u_g^*$, $z_g := \rE_A(u_g) = \mathbf 1_{X_g}$ and $u_g^* z_g = z_g u_g^* = z_g$. 
Since $A$ is abelian and hence tracial, a combination of the fact that $u_g^* z_g \in A$ and the Claim implies that 
$$\|\rE_A(u_n u_g^*) z_g \|_\varphi = \|\rE_A(u_n \, u_g^* z_g) \|_\varphi = \|\rE_A(u_n) \, u_g^* z_g \|_\varphi \leq \|\rE_A(u_n) \|_\varphi  \to 0 \quad \text{as} \quad n \to \infty.$$
Denote by $\mathcal J$ the nonempty directed set (for the inclusion) of all the families of projections $(z_i)_{i \in I}$ in $A$ such that $z_i \leq 1 - z_g$, $z_i \perp z_j$ for all $i \neq j \in I$ and $z_i \perp u_g z_j u_g^*$ for all $i, j \in I$. By Zorn's lemma, let $(z_i)_{i \in I}$ be a maximal element in $\mathcal J$. Put $z = \sum_{i \in I} z_i$ and assume by contradiction that $z + u_g z u_g^* \neq 1 - z_g$. Put $e = 1 - z_g - z - u_gz u_g^* \neq 0$. Since $g^2=1$, we have $u_geu_g^*=e$. Since $e\leq 1-z_g=\mathbf 1_{\{s \in X \mid g\cdot s\neq s\}}$, we can find $0 \neq z' \leq e$ such that $z' \perp u_g z' u_g^*$. Then the family $((z_i)_{i \in I}, z')$ is in $\mathcal J$ and this contradicts the maximality of the family $(z_i)_{i \in I}$ in $\mathcal J$. Therefore, we have $z + u_g z u_g^* = 1 - z_g$. A calculation entirely analogous to \cite[Proposition C, Equation (6.6)]{Ho15} shows that
\begin{align*}
\|\rE_A(u_n u_g^*)(1 - z_g)\|_\varphi^2 &= \|\rE_A(u_n u_g^*)(z + u_g z u_g^*)\|_\varphi^2 \\
&= \|\rE_A(u_n u_g^*)z\|_\varphi^2 + \|\rE_A(u_n u_g^*)u_g z u_g^*\|_\varphi^2 \\
&= \|\rE_A(u_n u_g^*)(z - u_g z u_g^*)\|_\varphi^2 \\
&= \|\rE_A((z u_n - u_n z) u_g^*)\|_\varphi^2.
\end{align*}
Since $zu_n - u_n z \to 0$ $\sigma$-strongly as $n \to \infty$, we also have that $\rE_A((z u_n - u_n z) u_g^*) \to 0$ $\sigma$-strongly as $n \to \infty$. The above calculation implies that $\lim_n \|\rE_A(u_n u_g^*)(1 - z_g)\|_\varphi = 0$. This further implies that 
$$\limsup_n \|\rE_A(u_n u_g^*) \|_\varphi^2 = \limsup_n \left(\|\rE_A(u_n u_g^*)z_g \|_\varphi^2 + \|\rE_A(u_n u_g^*)(1 - z_g) \|_\varphi^2\right) = 0.$$

Define $\mathcal F = \spn \{ au_g \mid a \in A, g \in [\mathcal R], g^2 = 1\}$. By the proof of \cite[Theorem 1]{FM75}, it follows that $\mathcal F$ is a $\sigma$-strongly dense linear $\ast$-subspace of $M$. The previous reasoning shows that $\lim_n \|\rE_A(u_n x)\|_\varphi = 0$ for every $x \in \mathcal F$. Let $\omega \in \beta(\N) \setminus \N$ be any nonprincipal ultrafilter. For every $x \in \mathcal F$, we have $\|\rE_{A^\omega}((u_n)^\omega x)\|_{\varphi^\omega} = \lim_{n \to \omega} \|\rE_A(u_n x)\|_\varphi = 0$ and hence $\rE_{A^\omega}((u_n)^\omega x) = 0$. Since $\mathcal F$ is $\sigma$-strongly dense in $M$, this further implies that $\rE_{A^\omega}((u_n)^\omega x) = 0$ for every $x \in M$. Using Condition $(\rm ii)$, we also have $\rE_{A^\omega}(x (u_n)^\omega y) = \rE_{A^\omega}((u_n)^\omega xy) = 0$ for every $x, y \in M$. This implies that $\lim_{n \to \omega} \|\rE_A(x u_n y)\|_\varphi = \|\rE_{A^\omega}(x (u_n)^\omega y)\|_{\varphi^\omega} = 0$. Since this is true for every $\omega \in \beta(\N) \setminus \N$, we finally obtain that $\lim_n \|\rE_A(x u_n y)\|_\varphi = 0$ for all $x, y \in M$.
\end{proof}

\begin{proof}[Proof of Theorem \ref{thmC}]
Simply write $B = \rL^\infty(X)$ and $M = B \rtimes \Gamma$. Assume by contradiction that $M$ is not full. Fix any nonprincipal ultrafilter $\omega \in \beta(\N) \setminus \N$. Since $\Gamma \curvearrowright (X, \mu)$ is strongly ergodic, Lemma \ref{lem-strong-ergodicity} shows that there exists $u \in \mathcal U(M_\omega)$ such that $\rE_{B^\omega}(u \lambda_s) = 0$ for every $s \in \Gamma$. Then for every $s \in \Gamma$, we have
$$\rE_{B^\omega}(\rE_{B^\omega \rtimes \Gamma}(u) \lambda_s) = \rE_{B^\omega}(\rE_{B^\omega \rtimes \Gamma}(u \lambda_s)) = \rE_{B^\omega}(u \lambda_s) = 0.$$
This implies that $\rE_{B^\omega \rtimes \Gamma}(u) = 0$. Since $M$ is a nonamenable factor, Theorem \ref{thmA} shows that $M_\omega \subset M' \cap M^\omega \subset B^\omega \rtimes \Gamma$ and hence $u \in \mathcal U(B^\omega \rtimes \Gamma)$. We then have $u = \rE_{B^\omega \rtimes \Gamma}(u) = 0$. This is a contradiction.
\end{proof}

\section{Group measure space type ${\rm III}$ factors with no central sequence}

\begin{df}
Let $G$ be any locally compact second countable group. Let $G \curvearrowright (X, \mu)$ and $G \curvearrowright (Y, \nu)$ be any nonsingular Borel actions on standard probability spaces. We say that 
\begin{itemize}
\item $G \curvearrowright (Y, \nu)$ is a {\em measurable quotient} of $G \curvearrowright (X, \mu)$ if, after discarding null $G$-invariant Borel subsets, there exists a $G$-equivariant Borel quotient map $q : X \to Y$ such that $[q_\ast \mu] = [\nu]$.
\item $G \curvearrowright (Y, \nu)$ is {\em measurably conjugate} to $G \curvearrowright (X, \mu)$ if, after discarding null $G$-invariant Borel subsets, there exists a $G$-equivariant Borel isomorphism $\theta : X \to Y$ such that $[\theta_\ast \mu] = [\nu]$.
\end{itemize}
\end{df}

Let $G$ be any locally compact second countable group and $H < G$ any closed subgroup. Endowed with the quotient topology, $G/H$ is a continuous $G$-space, that is, the action $G \curvearrowright G/H$ defined by $(g, hH) \mapsto ghH$ is continuous. The quotient space $G/H$ carries, up to equivalence, a unique $G$-quasi-invariant regular Borel probability measure $\nu \in \Prob(G/H)$. Any such $G$-quasi-invariant regular Borel probability measure is associated with a rho-function for the pair $(G, H)$ (see e.g.\ \cite[Appendix B]{BdlHV08}). The action $G \curvearrowright G/H$ is indeed a measurable quotient of the translation action $G \curvearrowright G$  (see \cite[Theorem B.1.4]{BdlHV08}).

Let $G$ be any noncompact connected simple Lie group and $P < G$ any minimal parabolic subgroup (e.g.\ $G = \SL_n(\R)$ and $P = $ subgroup of upper triangular matrices, for $n \geq 2$). Fix a $G$-quasi-invariant Borel regular probability measure $\nu \in \Prob(G/P)$. Denote by $\Delta_P : P \to \R^*_+$ the modular homomorphism and observe that $\Delta_P(P) = \R^*_+$ (this follows from  \cite[Proposition B.1.6 (ii)]{BdlHV08} and \cite[Proposition 4.3.2]{Zi84}). Put $L = \ker(\Delta_P)$. The Radon-Nikodym cocycle associated with the action $G \curvearrowright G/P$ is the map defined by 
$$\Omega : G \times G/P \to \R : (g, hP) \mapsto \log \left(\frac{{\rm d} g_\ast\nu}{{\rm d}\nu}(hP)\right).$$
Observe that $\Omega : G \times G/P \to \R$ is a continuous map by \cite[Theorem B.1.4]{BdlHV08}. The Maharam extension $G \curvearrowright G/P \times \R$ is the continuous action defined by 
$$g \cdot (hP, t) = (ghP, t + \Omega(g, hP)).$$
By \cite[Lemma B.1.3]{BdlHV08}, we have $\Omega(g, P) = (\log \circ \Delta_P)(g)$ for every $g \in P$. Since moreover $(\log \circ \Delta_P)(P) = \R$, the Maharam extension $G \curvearrowright G/P \times \R$ is transitive and the stabilizer of the point $(P, 0)$ is equal to $L$. The mapping 
$$\theta : G/L \to G/P \times \R : gL \mapsto (gP, \Omega(g, P))$$
is a well-defined $G$-equivariant homeomorphism that yields a measurable conjugacy between the  action $G \curvearrowright G/L$ and the Maharam extension $G \curvearrowright G/P \times \R$. Therefore, we have proved the following useful fact.

\begin{prop}[{see \cite[Proposition 4.7]{BN11}}]\label{proposition-maharam}
The Maharam extension of $G \curvearrowright G/P$ is measurably conjugate to $G \curvearrowright G/L$.
\end{prop}

From now on, fix $n \geq 2$, $G = \SL_n(\R)$ and $\Lambda = \SL_n(\Q)$ and denote by $P < G$ the minimal subgroup of upper triangular matrices. By \cite[Theorem A and Proposition 7.4]{BISG15}, the translation action $\Lambda \curvearrowright G$ is strongly ergodic and so is the nonsingular action $\Lambda \curvearrowright G/P$ (recall that strong ergodicity is stable under taking measurable quotients). Fix a surjective group homomorphism $\pi : \F_\infty \to \Lambda$ such that $\ker(\pi) < \F_\infty$ is a nonamenable subgroup.  For simplicity, write $\Gamma = \F_\infty$.

Put $X = [0, 1]^\Gamma$ and $\mu = \Leb^{\otimes \Gamma}$ and consider the Bernoulli shift action $\Gamma \curvearrowright X$ defined by $\gamma \cdot (x_{\gamma'})_{\gamma' \in \Gamma} = (x_{\gamma^{-1}\gamma'})_{\gamma' \in \Gamma}$. Observe that $\Gamma \curvearrowright X$ preserves the Borel probability measure $\mu$ and is essentially free and strongly ergodic. Since $\ker(\pi)$ is nonamenable, the restricted action $\ker(\pi) \curvearrowright X$ is also strongly ergodic. Since $\ker(\pi) < \F_\infty$ is a nonamenable free subgroup and hence not inner amenable, the crossed product ${\rm II_1}$ factor $\rL^\infty(X) \rtimes \ker(\pi)$ is full by \cite{Ch81}. Define the  action $\Gamma \curvearrowright X \times G/P$ by 
$$\gamma \cdot (x, hP) = (\gamma x, \pi(\gamma) hP).$$
Observe that $\Gamma \curvearrowright X \times G/P$ quasi-preserves the product measure $\mu \otimes \nu$ and is essentially free.

\begin{thm}\label{thm-examples}
Keep the same notation as above. The following assertions hold true:
\begin{itemize}
\item [$\rm(i)$] The nonsingular action $\Gamma \curvearrowright X \times G/P$ is essentially free and strongly ergodic and its Maharam extension $\Gamma \curvearrowright X \times G/P \times \R$ is also essentially free and strongly ergodic.
\item [$\rm(ii)$] The group measure space factor $M = \rL^\infty(X \times G/P) \rtimes \Gamma$ is a full type ${\rm III_1}$ factor and its continuous core $\core(M)$ is a full type ${\rm II_\infty}$ factor.
\end{itemize}
\end{thm}

\begin{proof}
$\rm(i)$ As we already pointed out, the nonsingular action $\Gamma \curvearrowright X \times G/P$ is essentially free and so is its Maharam extension $\Gamma \curvearrowright X \times G/P \times \R$. 

We next prove that the nonsingular action $\Gamma \curvearrowright X \times G$ defined by $\gamma \cdot (x, h) = (\gamma x, \pi(\gamma) h)$ is strongly ergodic. Put $A = \rL^\infty(X)$ and $B = \rL^\infty(G)$ so that $\rL^\infty(X \times G) = A \ovt B$. Write $N = (A \ovt B) \rtimes \Gamma$. Fix a nonprincipal ultrafilter $\omega \in \beta(\N) \setminus \N$. We need to show that $N' \cap (A \ovt B)^\omega = \C 1$. Observe that $(A \ovt B) \rtimes \ker(\pi) = (A \rtimes \ker(\pi)) \ovt B$ and $N' \cap (A \ovt B)^\omega \subset (A \rtimes \ker(\pi))' \cap ((A \rtimes \ker(\pi)) \ovt B)^\omega$. Since $A \rtimes \ker(\pi)$ is a full type ${\rm II_1}$ factor, \cite[Theorem 2.1]{Co75} implies that $(A \rtimes \ker(\pi))' \cap ((A \rtimes \ker(\pi)) \ovt B)^\omega = B^\omega$ and hence $N' \cap (A \ovt B)^\omega = N' \cap B^\omega = (B^\omega)^{\Lambda}$. By \cite[Theorem A]{BISG15}, the nonsingular action $\Lambda \curvearrowright G$ is strongly ergodic, that is, $(B^\omega)^{\Lambda} = \C 1$. This implies that $N' \cap (A \ovt B)^\omega = \C 1$ and hence the nonsingular action $\Gamma \curvearrowright X \times G$ is strongly ergodic.

Since the nonsingular action $\Gamma \curvearrowright X \times G/P$ is a quotient of the strongly ergodic nonsingular action $\Gamma \curvearrowright X \times G$, it follows that $\Gamma \curvearrowright X \times G/P$ is also strongly ergodic. Consider the Maharam extension $\Lambda \curvearrowright G/P \times \R$ of the nonsingular action $\Lambda \curvearrowright G/P$. By Proposition \ref{proposition-maharam}, the Maharam extension $\Lambda \curvearrowright G/P \times \R$ is measurably conjugate to the nonsingular action $\Lambda \curvearrowright G/L$ where $L = \ker(\Delta_P)$ and $\Delta_P : P \to \R^*_+$ is the modular homomorphism. Since $\Gamma \curvearrowright X$ is pmp, the  action $\Gamma \curvearrowright X \times G/L$ defined by $\gamma \cdot (x, hL) = (\gamma x, \pi(\gamma)hL)$ can be identified with the Maharam extension of the nonsingular  action $\Gamma \curvearrowright X \times G/P$. Since the nonsingular action $\Gamma \curvearrowright X \times G/L$ is a quotient of the strongly ergodic nonsingular action $\Gamma \curvearrowright X \times G$, it follows that $\Gamma \curvearrowright X \times G/L$ is also strongly ergodic. Therefore, the Maharam extension $\Gamma \curvearrowright X \times G/P \times \R$ of the nonsingular action $\Gamma \curvearrowright X \times G/P$ is strongly ergodic.

$\rm(ii)$ This is a consequence of Theorem \ref{thmC}.
\end{proof}

Keep the same notation as above. Fix $0 < \lambda < 1$, put $T = \frac{2 \pi}{|\log \lambda|}$ and identify $\mathbf T = \R / (T \Z)$. Define the nonsingular  action $\Gamma \curvearrowright X \times G/P \times \mathbf T$ by 
$$\gamma \cdot (x, hP, t + T \Z) = (\gamma x, \pi(\gamma) hP, t + \Omega(\pi(\gamma), hP) + T \Z).$$
Observe that the nonsingular action $\Gamma \curvearrowright X \times G/P \times \mathbf T$ is a measurable quotient of the nonsingular action $\Gamma \curvearrowright X \times G/P \times \mathbf R$ and hence is strongly ergodic by Theorem \ref{thm-examples}$\rm(i)$. Moreover, we have a canonical identification 
$$\rL^\infty(X \times G/P \times \mathbf T) \rtimes \Gamma = M \rtimes_{\sigma^\varphi_T} \Z.$$
It follows that $\rL^\infty(X \times G/P \times \mathbf T) \rtimes \Gamma$ is a type ${\rm III_\lambda}$ factor by \cite[Lemma 1]{Co85}. Observe that $\rL^\infty(X \times G/P \times \mathbf T) \rtimes \Gamma$ is full by Theorem \ref{thmC}. Alternatively, since $\core(M)$ is full by Theorem \ref{thm-examples}$(\rm ii)$, $\rL^\infty(X \times G/P \times \mathbf T) \rtimes \Gamma = M \rtimes_{\sigma^\varphi_T} \Z$ is full by \cite[Lemma 6]{TU14}.

\begin{proof}[Proof of Corollary \ref{corD}]
This is a consequence of Theorem \ref{thm-examples} and the above construction.
\end{proof}

\section{Further remarks}

	In \cite{HR14}, the first named author and Raum investigated the asymptotic structure of Shlyakhtenko's free Araki--Woods factors \cite{Sh96}. Among other things, they proved in \cite[Theorem A]{HR14} that any diffuse von Neumann algebra $M$ with separable predual satisfying Ozawa's condition (AO) is $\omega$-{\em solid}, that is, for any von Neumann subalgebra with expectation $N \subset M$ such that the relative commutant $N' \cap M^\omega$ is diffuse, we have that $N$ is amenable. The proof was based on a combination of Ozawa's C$^*$-algebraic techniques and an analysis of the relative commutant $N' \cap M^\omega$ and its centralizer \cite[Theorem 2.3]{HR14} (see \cite[Lemma 2.7]{Io12a} for the tracial case). 
	
In this subsection, we observe that $\omega$-solidity can be easily obtained using the same proof as the one of Theorem \ref{theorem for thmA 1} without relying on the analysis of the relative commutant $N' \cap M^\omega$ from \cite[Theorem 2.3]{HR14}. We moreover remove the separability assumption of the predual.

\begin{thm}[{\cite[Theorem A]{HR14}}]
	Let $M$ be any diffuse $\sigma$-finite von Neumann algebra satisfying Ozawa's condition $\rm (AO)$. Let $p\in M$ be any nonzero projection and $N\subset pMp$ any von Neumann subalgebra with expectation. 
	
	Then at least one of the following conditions holds true:
\begin{itemize}
	\item The von Neumann algebra $N$ has a nonzero amenable direct summand.
	\item We have $N'\cap pM^\omega p \subset  pMp$. In that case, $N'\cap pM^\omega p = N' \cap pMp$ is moreover discrete.
\end{itemize}
\end{thm}
\begin{proof}
	Suppose that $N$ has no amenable direct summand. Then the exact same argument as in the proof of Theorem \ref{theorem for thmA 1} using Ozawa's condition (AO) in lieu of Proposition \ref{AO in crossed product} shows that $N'\cap pM^\omega p\subset pMp $ and hence $N'\cap pM^\omega p= N'\cap pMp $. Since $M$ is diffuse and solid \cite[Theorem 6]{Oz03} (see also \cite[Theorem 2.5]{VV05}), it follows that $p M p$ is also diffuse and solid. Since $N \subset pMp$ has no amenable direct summand, it follows that $N' \cap pMp$ is necessarily discrete.
\end{proof}

In view of Proposition \ref{AO in ultraproduct}, we finally observe the following condition (AO) in the ultraproduct representation. 
\begin{prop}
	Let $M$ be any $\sigma$-finite von Neumann algebra and $\omega \in \beta(\N) \setminus \N$ any nonprincipal ultrafilter. Denote by $(M, \rL^2(M), J^M, \mathfrak P^M)$ (resp.\ $(M^\omega, \rL^2(M^\omega), J^{M^\omega}, \mathfrak P^{M^\omega})$) a standard form for $M$ (resp.\ $M^\omega$). Assume there are unital $\rC^*$-subalgebras $A, B\subset M$ such that the map
$$\nu : A \ota J^MBJ^M \to \B(\rL^2(M))/\K(\rL^2(M)) : a \otimes J^MbJ^M \mapsto  a \, J^MbJ^M + \K(\rL^2(M))$$ 
is continuous with respect to the minimal tensor norm. 

Then the map 
$$\nu_\omega : A \otimes_{\rm alg} J^{M^\omega}B J^{M^\omega} \rightarrow \B(\rL^2(M^\omega))/\K(\rL^2(M^\omega)) :  a \otimes J^{M^\omega} b J^{M^\omega} \mapsto a \, J^{M^\omega}bJ^{M^\omega} + \K(\rL^2(M^\omega))$$
is continuous with respect to the minimal tensor norm.
\end{prop}

\begin{proof}
The proof is a variation of the one of Proposition \ref{AO in ultraproduct}. Put 
\begin{align*}
C &:=\rC^* \left\{ M, J^MM J^M \right\} \subset \B(\rL^2(M)) \\
 C_\omega &:=\rC^* \left\{ M, J^{M^\omega}M J^{M^\omega} \right\} \subset \B(\rL^2(M^\omega)).
\end{align*}
Observe that $C + \mathbf K(\rL^2(M))$ (resp.\ $C_\omega + \mathbf K(\rL^2(M^\omega))$) is a $\rC^*$-subalgebra of $\mathbf B(\rL^2(M))$ (resp.\ $\mathbf B(\rL^2(M^\omega))$). Fix a faithful state $\varphi \in M_\ast$. Denote by $e : \rL^2(M) \to \C \xi_{\varphi}$ and $f : \rL^2(M^\omega) \to \C \xi_{\varphi^\omega}$ the corresponding orthogonal projections. Observe that $\mathbf K(\rL^2(M))$ is the norm closure in $\mathbf B(\rL^2(M))$ of $M e M$. Denote by $N := \prod^\omega M$ the Groh--Raynaud ultraproduct and by $p \in N$ the support projection of the ultraproduct state $\varphi_\omega \in N_\ast$.

\begin{claim}
There is a $\ast$-homomorphism $\theta : C + \mathbf K(\rL^2(M)) \to \mathbf B(\rL^2(M^\omega))$ such that $\theta(x) = x$ and $\theta(J^M y J^M) = J^{M^\omega} y J^{M^\omega}$ for all $x, y \in M$ and $\theta(e) = f$. 
\end{claim}

\begin{proof}[Proof of the Claim]
Keep the same notation as in the proof of the Claim of Proposition \ref{AO in ultraproduct}. By Lemma \ref{lemma for AO in ultraproduct}, $\pi^\omega((e)_n)$ commutes with $p$ and $J^N$ and hence $\pi^\omega((e)_n)$ commutes with $\widetilde p := p J^N p J^N$. Since $\widetilde p$ commutes with $\pi^\omega(M)$ and $\pi^\omega(J^M M J^M) = J^N \pi^\omega(M) J^N$, $\widetilde p$ commutes with $\pi^\omega(C + \mathbf K(\rL^2(M)))$. Recall that $\widetilde p N \widetilde p \cong p N p \cong M^\omega$ and $\widetilde p \rL^2(M)_\omega = \rL^2(M^\omega)$. Then the $\ast$-homomorphism
$$\theta : C + \mathbf K(\rL^2(M)) \to \mathbf B(\rL^2(M^\omega)) : T \mapsto \widetilde p \pi^\omega(T) \widetilde p$$
satisfies all the conditions of the Claim.
\end{proof}

Since $\theta(C)=C_\omega$ and $\theta(\mathbf K(\rL^2(M))) \subset \mathbf K(\rL^2(M^\omega))$, $\theta$ induces a $\ast$-homomorphism 
$$\widetilde{\theta} : \left( C + \mathbf K(\rL^2(M)) \right)/\mathbf K(\rL^2(M)) \to \left( C_\omega + \mathbf K(\rL^2(M^\omega)) \right)/\mathbf K(\rL^2(M^\omega)).$$ 
Denote by $\iota : A \otimes_{\min} J^{M^\omega}B J^{M^\omega} \to A \otimes_{\min} J^{M}BJ^{M}$ the tautological $\ast$-isomorphism. Then the composition map 
$$\nu_\omega = \widetilde{\theta}\circ \nu \circ \iota : A \otimes_{\alg} J^{M^\omega}BJ^{M^\omega} \rightarrow \B(\rL^2(M^\omega))/\K(\rL^2(M^\omega))$$
is continuous with respect to the minimal tensor norm. 
\end{proof}


\end{document}